\documentclass[11pt]{article}
\usepackage{amsmath, amssymb, amsthm, verbatim,enumerate,bbm}
\usepackage{indentfirst}

\date{}
\title{\vspace{-0.9cm}Finding Hamilton cycles in random graphs with few queries}
\author{
Asaf Ferber \thanks{Department of Mathematics, Yale University, and
Department of Mathematics, MIT. Emails: asaf.ferber@yale.edu, and
ferbera@mit.edu.} \and Michael Krivelevich \thanks{School of
Mathematical Sciences, Raymond and Beverly Sackler Faculty of Exact
Sciences, Tel Aviv University, Tel Aviv, 6997801, Israel. Email:
krivelev@post.tau.ac.il. Research supported in part by USA-Israel
BSF Grant 2010115 and by grant 912/12 from the Israel Science
Foundation.}\and Benny Sudakov
\thanks{Department of Mathematics, ETH, 8092 Zurich, Switzerland.
Email: benjamin.sudakov@math.ethz.ch. Research supported in part by
SNSF grant 200021-149111.} \and Pedro Vieira \thanks{Department of
Mathematics, ETH, 8092 Zurich, Switzerland.  Email:
pedro.vieira@math.ethz.ch.} }

\allowdisplaybreaks

\theoremstyle{plain}
\newtheorem{theor}{Theorem}
\newtheorem{theorem}{Theorem}[section]
\newtheorem{lemma}[theorem]{Lemma}
\newtheorem{claim}[theorem]{Claim}
\newtheorem{proposition}[theorem]{Proposition}

\newtheorem{definition}[theorem]{Definition}

\newcommand{\Bin}{\ensuremath{\textrm{Bin}}}

\begin{document}
\maketitle
\begin{abstract}
We introduce a new setting of algorithmic problems in random
graphs, studying the minimum number of queries one needs to ask about the adjacency between pairs of vertices of ${\mathcal G}(n,p)$ in order to typically find a
subgraph possessing a given target property. We show that if $p\geq
\frac{\ln n+\ln\ln n+\omega(1)}{n}$,
then one can find a Hamilton cycle with high probability after exposing $(1+o(1))n$ edges.
Our result is tight in both $p$ and the number of exposed edges.
\end{abstract}

\section{Introduction}

Random Graphs is definitely one of the most popular areas in modern
Combinatorics, also having a tremendous amount of applications in
different scientific fields such as Networks, Algorithms, Communication,
Physics, Life Sciences and more. Ever since its introduction, the
\emph{binomial random graph model} has been one of the main objects
of study in probabilistic combinatorics. Given a positive integer
$n$ and a real number $p \in [0,1]$, the binomial random graph
$\mathcal G(n,p)$ is a probability space whose ground set consists
of all labeled graphs on the vertex set $[n]$. We can describe the
probability distribution of $G\sim \mathcal G(n,p)$ by saying that
each pair of elements of $[n]$ forms an edge in $G$ independently
with probability $p$. For more details about random graphs the
reader is referred to the excellent books of Bollob\'as
\cite{bollobas1998random} and of Janson, \L uczak and Ruci\'nski
\cite{JansonLuczakRucinski}.

Due to the importance and visibility of the subject of Random
Graphs, and also due to its practical connections and the fact that
random discrete spaces are frequently used to model real world
phenomena, it is only natural to study the algorithmic aspects of
random graphs. The reader is advised to consult an excellent survey
of Frieze and McDiarmid on the subject \cite{FM97}, providing an
extensive coverage of the variety of problems and results in
Algorithmic Theory of Random Graphs. In this paper we present an
apparently new and interesting setting for algorithmic type
questions about random graphs.

Usually, questions considered in random graphs have the following
generic form: given some \emph{monotone increasing graph property
${\cal P}$} (that is, a property of graphs that cannot be violated
by adding edges) and a function $p = p(n) \in [0,1]$, determine
whether a graph $G\sim \mathcal G(n,p)$ satisfies ${\cal P}$ {\em
with high probability} (whp) (that is, with probability tending to
$1$ as $n$ tends to infinity). In order to solve questions of this
type, one should show that after asking for all possible pairs
$(i,j)$ of distinct elements of $[n]$ the question ``is $(i,j)\in
E(G)$?" and getting a positive answer with probability $p(n)$
independently, whp the graph $G$ obtained from all positive answers
possesses $\mathcal P$. Here we propose a different task. Given $p$
such that a graph $G\sim \mathcal G(n,p)$ whp satisfies $\mathcal
P$, we want to devise an algorithm, probably an adaptive one, that
asks typically as few queries ``is $(i,j)\in E(G)$?" as possible,
and yet the positive answers reveal us a graph which possesses
$\mathcal P$. We refer to such an algorithm as \emph{an adaptive
algorithm interacting with the probability space $\mathcal G(n,p)$}.
For example, consider the case where $\mathcal P$ is the property
``containing a Hamilton cycle" (i.e. a cycle passing through all the
vertices of the graph). In this case we aim to find an algorithm
that will adaptively query as few pairs as possible, yet a
sufficient amount to get whp a Hamilton cycle between the positive
answers. It is important to remark that we are not concerned here
with the time of computation required for the algorithm to locate a
target structure (thus essentially assuming that the algorithm has
unbounded computational power), but we make the algorithm pay for
the number of queries it asks, or for the amount of communication
with the random oracle generating the random graph. Therefore, in
this sense this setting is reminiscent of such branches of Computer
Science as Communication Complexity and Property Testing.

In general, given a monotone property $\mathcal P$, what can we
expect? If all $n$-vertex graphs belonging to ${\cal P}$ have at
least $m$ edges, then the algorithm should get at least $m$ positive
answers to hit the target property with the required absolute
certainty. This means that the obvious lower bound in this case is
$(1+o(1))m/p$ queries, and therefore, whp one would have $(1+o(1))m$
positive answers. Continuing with our example of Hamiltonicity, this
lower bound translates to $(1+o(1))n$ positive answers. This should
serve as a natural benchmark for algorithms of such type. Of course,
the above described framework is very general and can be fed with
any monotone property ${\cal P}$, thus producing a variety of
interesting questions.

Here is a very simple illustration of our model. Let us choose the target property to be connectedness (i.e., the existence of a spanning tree) in $G\sim\mathcal G(n,p))$. Suppose the edge probability $p(n)$ is chosen to be above the threshold for connectedness, which is known to be equal to $p(n)=\frac{\ln n+\omega(1)}{n}$. In this case the minimum number of positive answers to the algorithm's queries is obviously $n-1$. An adaptive algorithm discovering a spanning tree after $n-1$ positive answers is very simple: start with $T={v}$, where $v\in [n]$ is an arbitrary vertex, and at each step query in an arbitrary order previously non-queried pairs between the current tree $T$ and the vertices outside of $T$ until the first such edge $(u,w)\in G$ has been found, then update $T$ by appending the edge $(u,w)$. Assuming the input graph $G$ is connected, the algorithm clearly creates a spanning tree of $G$ after exactly $n-1$ positive answers.

In this paper we focus on the property of Hamiltonicity, which is
one of the most central notions in graph theory, and has been
intensively studied by numerous researchers. The earlier results on
Hamiltonicity of random graphs were proved by Korshunov
\cite{korshunov1976solution} and by P\'osa
\cite{posa1976hamiltonian}  in 1976. Building on these ideas,
Bollob\'as \cite{bollobas1984evolution}, and Koml\'os and
Szemer\'edi \cite{komlos1983limit} independently showed that for
$p\geq \frac{\ln n+\ln\ln n+\omega(1)}{n}$, a graph $G\sim \mathcal
G(n,p)$ is whp Hamiltonian. This range of $p$ cannot be further
improved since if $p\leq \frac{\ln n+\ln\ln n-\omega(1)}{n}$, then
whp a graph $G\sim \mathcal G(n,p)$ has a vertex of degree at most
$1$, and such a graph is trivially non-Hamiltonian.

In the following theorem, which is the main result of this paper, we
verify what the reader may have suspected: $(1+o(1))n$ positive
answers (and thus, $(1+o(1))n/p$ queries) are enough to create a
graph which contains a Hamilton cycle, for every $p\geq \frac{\ln
n+\ln\ln n+\omega(1)}{n}$.

\begin{theor}\label{thm:main1}
Let $p=p(n)\ge\frac{\ln n+\ln\ln n +\omega(1)}{n}$. Then there exists
an adaptive algorithm, interacting with the probability space
$\mathcal G(n,p)$, which whp finds a Hamilton cycle after getting
$(1+o(1))n$ positive answers.
\end{theor}

Note that Theorem \ref{thm:main1} is asymptotically optimal in both
the edge probability $p$ and the number of positive answers we get.
We remark that if we allow (say) $3n$ positive answers, and replace
$p(n)\ge\frac{\ln n+\ln\ln n +\omega(1)}{n}$ with $p(n)\geq
(1+\varepsilon)\ln n/n$, then the result follows quite easily from a
result of Bohman and Frieze \cite{BF} by effectively embedding some
other random space, like a $3$-out, in $\mathcal G(n,p)$ and accessing
it in few queries. Moreover, if our goal is to find a Hamilton cycle after having $(1+o(1))n$ positive answers but we are willing to weaken the assumption to $p=\omega(\log n/n)$, then in fact the problem becomes much easier. In what follows, as an illustrative example, we provide a sketch of a proof for this statement.
\begin{proposition}
 Let $p=\omega(\log n/n)$. Then there exists an adaptive algorithm, interacting with the probability space $\mathcal G(n,p)$, which whp finds a Hamilton cycle after getting $(1+o(1))n$ positive answers.
\end{proposition}

\begin{proof}[Proof (sketch).] Let $p=f(n)\log n/n$, where $f:=f(n)$ is an arbitrary function tending to infinity with $n$, and set $q_1 = (1-\varepsilon)p$ (where, say, $\varepsilon = f^{-1/3}$) and $q_2$ to be the unique solution to $1-p=(1-q_1)(1-q_2)^4$.
Our proof consists of two phases, where in Phase 1 we find a ``long" path, and in Phase 2 we close it into a Hamilton cycle.

{\bf Phase 1} In this phase we construct a path $P$ of length $t:=n-m$ where $m= n/\sqrt{f}$, while exposing edges in an ``online fashion" with edge probability $q_1$, after exposing (successfully) exactly $t$ edges. At each time step $0\leq \ell\leq t-1$ of this phase we try to extend a current path $P_\ell:=v_0\ldots v_\ell$ by exposing an edge of the form $v_\ell u$, where $u\in [n]\setminus V(P_\ell)$. Note that the probability to fail in step $\ell$ is $(1-q_1)^{n-(\ell+1)}\leq e^{-q_1m}=o(1/n)$,
and therefore, by applying the union bound, whp the algorithm does not terminate unsuccessfully during this phase.

{\bf Phase 2} In this phase we want to turn the path $P:=v_0\ldots v_{t}$ obtained in Phase 1 into a Hamilton cycle. To this end, we define an auxiliary directed graph $D$, based on a subgraph of $G$, and show that a \emph{directed} Hamilton cycle in this graph exists whp, and that such a cycle corresponds to a Hamilton cycle of $G$. Moreover, we show that $D$ contains $O(km)=o(n)$ edges. This will complete the proof.

Let $U:=\left([n]\setminus V(P)\right)\cup \{v_P\}$ and set $V(D)=U$. Let us choose (say) $k:=100$ (actually, any $k\geq 2$ will work, but for a large $k$ the proof of the tool we base our argument on becomes relatively simple).
We choose the \emph{arcs} (directed edges) of $D$ according to the following procedure:
\begin{itemize}
  \item For every $v\in U\setminus \{v_P\}$, we define $$\text{Out}(v)=(U-v)\cup \{v_0\}$$ and $$\text{In}(v)=(U-v)\cup\{v_{t}\}.$$
  Now, in $G$, iteratively expose edges of the form $vu$, $u\in \text{Out}(v)$, with probability $q_2$, independently at random, according to a \emph{random} ordering of $\text{Out}(v)$, until you have \emph{exactly} $k$ successes. Let $x_1,x_2,\ldots,x_k$ denote these successes, and add $vx_i$, $1\leq i\leq k$ as arcs to $E(D)$ (if one of the $x_i$'s is $v_{0}$, then add the arc $vv_P$ to $E(D)$).
  Do the same for edges of the form $uv$, $u\in \text{In}(v)$, where an arc of the form $v_{t}v$ translates to the arc $v_Pv$.
  \item For $v_P$, let $\text{Out}(v_P)=\text{In}(v_P):=U\setminus \{v_P\}$. Expose (in $G$) edges of the form $v_{t}u$, $u\in \text{Out}(v_P)$ with probability $q_2$, independently at random, according to a random ordering of $\text{Out}(v_P)$, until having exactly $k$ successes. For each success $v_{t}u$, add an arc $v_Pu$ to $E(D)$.
      Do the same for $\text{In}(v_P)$, where now one exposes edges of the form $uv_{0}$, $u\in \text{In}(v_P)$. For each edge of the form $uv_{0}$ add the arc $uv_P$ to $E(D)$.
\end{itemize}

It is relatively easy to show that the probability for not having $k$ successes for at least one of the vertices is $o(1)$, and that the choices are made independently and uniformly. Moreover, the number of successively exposed edges is clearly $O(km)= o(n)$.
Therefore, $D$ can be seen as a directed graph obtained by the following procedure: each vertex picks exactly $k$ in- and $k$ out-neighbors, independently, uniformly at random. This model is known as $D_{k-in,k-out}$ and whp contains a directed Hamilton cycle (see the main result of \cite{CF}, or the less complicated one in \cite{CFr}). Clearly, by replacing the vertex $v_P$ by the path $P$ itself, such a Hamilton cycle corresponds to a Hamilton cycle of $G$. Moreover, note that throughout the algorithm (both phases), each edge has been queried at most once with probability $q_1$ (in Phase 1) and at most $4$ times with probability $q_2$ (in Phase 2), and therefore the resulting graph naturally couples as a subgraph of $G\sim \mathcal G(n,p)$. It's not hard to see that if we were to expose with full probability $p$ the edges that have been exposed by this procedure with any non-zero probability, the number of additional successively exposed edges would be $o(n)$, because of our choices of $q_1$ and $q_2$. This completes the proof of this sketch.
\end{proof}

\section{Notation}
Most of the notation used in this paper is fairly standard. Given a natural number $k$ and a set $S$, we use $[k]$ to denote the set $\{1,2,\ldots,k\}$ and $\binom{S}{k}$ to the denote the collection of subsets of $S$ of size $k$.

Given a graph $G$ we use $V(G)$ to denote the set of vertices of $G$. Moreover, given a subset $E$ of the edges of $G$ we shall oftentimes abuse notation and refer to the subgraph of $G$ formed by these edges simply by $E$ (with vertex set $V(G)$ unless stated otherwise).

Given a subset $S\subseteq V(G)$, $G[S]$ denotes the subgraph of $G$ induced by the vertices in $S$, i.e. the graph with vertex set $S$ whose edges are the ones of $G$ between vertices in $S$. Furthermore, we use $e_G(S)$ to denote the number of edges of the graph $G[S]$.

Given a vertex $v\in V(G)$ and a subset $S\subseteq V(G)$, we use $N_G(v)$ to denote the set of neighbours of $v$ in the graph $G$, $N_G(S):=\cup_{v\in S}N_G(v)$ to denote the set of neighbours of vertices in $S$ in the graph $G$ and $N_G(v,S):=N_G(v)\cap S$ to denote the set of neighbours of $v$ in the graph $G$ which lie in the set $S$. Moreover, $d_G(v):=|N_G(v)|$ denotes the degree of $v$ in the graph $G$, $d_G(v,S):=|N_G(v,S)|$ denotes the number of neighbours of $v$ in the graph $G$ which lie in the set $S$, $\Delta(G):=\max_{v\in V(G)}d_G(v)$ denotes the maximum degree of the graph $G$ and finally $\delta(G):=\min_{v\in V(G)}d_G(v)$ denotes the minimum degree of the graph $G$.

A subgraph $P$ of the graph $G$ is called a \textit{path} if $V(P)=\{v_1,\ldots,v_{\ell}\}$ and the edges of $P$ are $v_1v_2$, $v_2v_3$, $\ldots$, $v_{\ell-1}v_{\ell}$. We shall oftentimes refer to $P$ simply by $v_1v_2\ldots v_{\ell}$. We say that such a path $P$ has \textit{length} $\ell-1$ (number of edges) and \textit{size} $\ell$ (number of vertices). We say that $P$ is a \textit{Hamilton path} (in the graph $G$) if it has size $|V(G)|$. Furthermore, a subgraph $C$ of the graph $G$ is called a \textit{cycle} if $V(C)=\{v_1,\ldots,v_{\ell}\}$ and the edges of $P$ are $v_1v_2$, $v_2v_3$, $\ldots$, $v_{\ell-1}v_{\ell}$ and $v_{\ell}v_1$. As for paths, we shall oftentimes refer to $C$ simply by $v_1v_2\ldots v_{\ell}$. We say that such a cycle has \textit{length} $\ell$ (number of edges) and \textit{size} $\ell$ (number of vertices). We say that $C$ is a \textit{Hamilton cycle} (in the graph $G$) if it has size $|V(G)|$. A \textit{trail} of length $t$ in $G$ between two vertices $x$ and $y$ is a sequence of vertices $x=v_0,v_1,\ldots,v_{t}=y$ such that $\{v_0v_1,v_1v_2,\ldots, v_{t-1}v_{t}\}$ is a set of distinct $t$ edges of $G$.

\section{Auxiliary results}

\subsection{Probabilistic tools}

We need to employ standard bounds on large deviations of random
variables. We mostly use the following well-known bound on the lower and upper tails of the Binomial distribution due to Chernoff
(see e.g. \cite{AlonSpencer}, \cite{JansonLuczakRucinski}).

\begin{lemma}\label{Che}
Let $X \sim \emph{\text{Bin}}(n,p)$ and let $\mu=\mathbb{E}[X]$.
Then
\begin{itemize}
    \item $\Pr\left[X\le(1-a)\mu\right]<e^{-a^2\mu/2}$ for every
    $a>0$;
    \item $\Pr\left[X\ge(1+a)\mu\right]<e^{-a^2\mu/3}$ for every $0<a<3/2.$
\end{itemize}
\end{lemma}

\noindent The following is a trivial yet useful bound.
\begin{lemma}\label{Che2}
Let $X \sim \emph{\Bin}(n,p)$ and $k \in \mathbb{N}$. Then the
following holds:
\[\Pr(X\geq k) \leq \left(\frac{enp}{k}\right)^k.\]
\end{lemma}
\begin{proof}
$\Pr(X \geq k) \leq \binom{n}{k}p^k \leq
\left(\frac{enp}{k}\right)^k$.
\end{proof}

\subsection{Properties of random graphs}

We start with the following natural definition for
$k$-pseudorandomness of graphs.
\begin{definition}
A graph $G$ is called \emph{$k$-pseudorandom} if
$e_{G}(A,B)>0$ for every two disjoint sets $A$ and $B$ of size at least $k$.
\end{definition}

\begin{lemma}\label{pseudorandom}
Let $k=k(n)$ be an integer such that $3\ln n\leq k\le \frac{n}{8}$ and let $\frac{3\ln \left(n/k\right)}{k}\leq p\leq 1$. Then whp a graph $G\sim \mathcal{G}(n,p)$ is $k$-pseudorandom.
\end{lemma}

\begin{proof}
If $G$ is not $k$-pseudorandom then there exist two disjoint
sets $S$ and $T$ with $|S|=|T|=k$ and no edge between them.
Note that the probability that there is no edge between a given such
pair $\{S,T\}$ is $(1-p)^{k^2}$ and there are at most
$\binom{n}{k}^2$ such pairs. Thus, applying the union bound over all pairs of disjoint sets $S,T$ of size $k$ we obtain that the probability that $G$ is not $k$-pseudorandom is at most
\begin{align*}
&\binom{n}{k}^2(1-p)^{k^2}\leq \left(\frac{en}{k}\right)^{2k}e^{-pk^2}=\left(\frac{e^2n^2e^{-pk}}{k^2} \right)^k\leq  \left(\frac{e^2k}{n}\right)^k=o(1).
\end{align*}
Thus, we conclude that whp a graph $G\sim \mathcal{G}(n,p)$ is $k$-pseudorandom as
claimed.
\end{proof}

In the following two lemmas we state a few properties of a typical random
graph which will be used extensively throughout the paper.

\begin{lemma}\label{lemma:propertiesofGnp}
Let $p=p(n)\in (0,1)$, let $c>1$ be a constant and let $C=C(n)\ge
6\ln (np\ln n)$. Then whp a graph $G\sim \mathcal G(n,p)$ is such
that the following holds:
\begin{enumerate}[$(P1)$]
\item $\Delta(G)\leq 4np$, provided $p\ge\frac{\ln n}{n}$.
\item $e_G(X)< c|X|$ for any subset $X\subseteq
V(G)$ of size at most $\left(\frac{1}{\ln n}\cdot\frac{(2c)^c}{e^{c+1}np^c}\right)^{\frac{1}{c-1}}$.
\item $e_G(X)< C|X|$ for any subset $X\subseteq
V(G)$ of size at most $\frac{C}{2p}$.
\end{enumerate}
\end{lemma}

\begin{proof} For $\text{(P1)}$, note that for a vertex $v\in V(G)$ we have $d_{G}(v)\sim \Bin(n-1,p)$ and so by Lemma \ref{Che2}
\[\Pr\left[d_G(v)\geq 4np\right]\leq \left(\frac{e(n-1)p}{4np}\right)^{4np}\leq \left(\frac{e}{4
}\right)^{4\ln n}=o\left(\frac 1n\right).\]
Thus, by applying the union bound over all vertices of $G$ we see that the probability that there exists a vertex $v\in V(G)$ with $d_G(v)\ge 4np$ is $o(1)$, settling $\text{(P1)}$.

For $\text{(P2)}$, note that for a fixed $X\subseteq V(G)$ of size $|X|=x$ one has $e_G(X)\sim \Bin\left(\binom{x}{2},p\right)$. Therefore, by Lemma \ref{Che2} we obtain
\[\Pr\left[e_G(X)\ge cx\right]\le \left(\frac{ex^2p}{2cx}\right)^{cx}=\left(\frac{exp}{2c}\right)^{cx}.\]
Applying the union bound over all subsets of $V(G)$ of size at most
$t=\left(\frac{1}{\ln n}\cdot\frac{(2c)^c}{e^{c+1}np^c}\right)^{\frac{1}{c-1}}$ we see that the probability
that there exists a set $X$ of size $x\le t$ with $e_G(X)\ge cx$
is upper bounded by
\begin{align*}
&\;\sum_{x=1}^{t}\binom{n}{x}\left(\frac{exp}{2c}\right)^{cx} \le \sum_{x=1}^{t}\left(\frac{en}{x}\right)^x\left(\frac{exp}{2c}\right)^{cx}\\
= &\;
\sum_{x=1}^{t}\left[\frac{en}{x}\cdot\left(\frac{exp}{2c}\right)^c\right]^{x}\le
\sum_{x=1}^{t}\left(\frac{e^{c+1}np^ct^{c-1}}{(2c)^c}\right)^{x}=o(1)\,,
\end{align*}
since $\frac{e^{c+1}np^ct^{c-1}}{(2c)^c}=\frac{1}{\ln n}$. This settles $\text{(P2)}$.

For $\text{(P3)}$ we proceed in a similar way as with $\text{(P2)}$. By $\text{(P2)}$, taking $c=2$, we know that whp all sets $X\subseteq V(G)$ of size at most $\frac{1}{2np^2\ln n}$ satisfy $e_G(X)< 2|X|\le C|X|$. Thus, we just need to show that the probability that there exists a set $X$ of size $\frac{1}{2np^2\ln n}\le x\le \frac{C}{2p}$ with $e_G(X)\ge Cx$ is $o(1)$. Indeed, proceeding as above, we see that this probability is at most
\begin{align*}
&\sum_{x=\frac{1}{2np^2\ln n}}^{
\frac{C}{2p}}\left[\frac{en}{x}\cdot\left(\frac{exp}{2C}\right)^C\right]^x
\le \sum_{x=1}^{
\infty}\left[2e(np)^2 \ln
n\left(\frac{e}{4}\right)^C\right]^x \le \sum_{x=1}^{
\infty}\left(\frac{2e}{\ln n}\right)^x = o(1).
\end{align*}
since $C\ge 6\ln (np\ln n)$ and $\left(\frac{e}{4}\right)^{6}<e^{-2}$. This settles $\text{(P3)}$.
\end{proof}

\begin{lemma}
\label{lemma: minimum degree 2 and no two small cycles are close}
Let $w:=w(n)$ be such that $w\rightarrow \infty$ as $n\rightarrow \infty$, let $p:=p(n)$ be such that $\frac{\ln n+\ln \ln n+w}{n}\le p\le \frac{10\ln n}{n}$, let $G\sim\mathcal{G}(n,p)$ and let $C$ be a fixed positive integer. Then whp all of the following hold:
\begin{enumerate}[$(P1)$]
\item $G$ has minimum degree at least $2$.
\item there are no two cycles in $G$ of length at most $\frac{\ln n}{4\ln (np)}$ sharing exactly one vertex.
\item between any two vertices $u,v\in V(G)$ there are at most $3$ trails of length at most $C$.
\end{enumerate}
\end{lemma}

\begin{proof}
First we prove that whp $(P1)$ holds. Note that for a fixed vertex $v\in V(G)$ we have
\[\Pr\left[d_{G}(v)<2\right]=(1-p)^{n-1}+(n-1)p(1-p)^{n-2}\le 2e^{-np}+2npe^{-np}\le \frac{2e^{-w}}{n\ln n}+20\ln n\cdot\frac{e^{-w}}{n\ln n}=o\left(\frac{1}{n}\right)\]
where in the second inequality we used the fact that $\ln n+\ln \ln n+w\le np\le 10\ln n$, and in the last step we used the fact that $w\rightarrow \infty$. Thus, taking the union bound over all vertices in $V(G)$ we obtain
\[\Pr\left[\delta(G)<2\right]=\Pr\left[\exists v\in V(G)\text{ such that }d_G(v)<2\right]\le n\cdot o\left(\frac{1}{n}\right)=o(1)\,,\]
implying that $(P1)$ holds whp as claimed.

Next we show that $(P2)$ also holds whp. Note that if $\mathcal C_1$ and $\mathcal C_2$ are cycles in $K_n$ of lengths $l_1,l_2$, respectively, sharing exactly one vertex then $|V(\mathcal C_1)\cup V(\mathcal C_2)|=l_1+l_2-1$ and $|E(\mathcal C_1)\cup E(\mathcal C_2)|=l_1+l_2$. Thus, using the union bound, we see that the probability that in $G$ we have two cycles of lengths $l_1$ and $l_2$ which share exactly one vertex is at most $n^{l_1+l_2-1}p^{l_1+l_2}=\frac{(np)^{l_1+l_2}}{n}$. Moreover, letting $k:=\frac{\ln n}{4\ln (np)}$ and taking the union bound over all pairs of cycles of lengths at most $k$ which share exactly one vertex, we see that the probability of $(P2)$ not holding is at most
\[\sum_{l_1,l_2=1}^{k}\frac{(np)^{l_1+l_2}}{n}\le \frac{k^2(np)^{2k}}{n}=o(1)\,,\]
since $(np)^{2k}=\sqrt{n}$. Thus, we conclude that $(P2)$ holds whp as claimed.

Finally we show that whp $(P3)$ holds. Let $u,v\in V(G)$ and let
$\mathcal{W}_1,\mathcal W_2,\mathcal W_3,\mathcal W_4$ be any four
distinct trails between $u$ and $v$ in $K_n$, each of length at most
$C$. A moment's thought reveals that
\[|V(\mathcal W_1)\cup V(\mathcal W_2)\cup V(\mathcal W_3)\cup V(\mathcal W_4)|< |E(\mathcal W_1)\cup E(\mathcal W_2)\cup E(\mathcal W_3)\cup E(\mathcal W_4)|\le 4C.\]
Thus, using the union bound we see that the probability of $(P3)$ not holding is at most
\[\sum_{l=1}^{4C}n^{l-1}p^{l}=\sum_{l=1}^{4C}\frac{(np)^{l}}{n}\le \frac{4C(10\ln n)^{4C}}{n}=o(1).\]
We conclude that $(P3)$ holds whp, completing the proof of the lemma.
\end{proof}

\subsection{Properties of graphs}
The following simple lemma can be found, e.g., in Chapter 1 of \cite{Diestel}.
\begin{lemma}
\label{lemma:subgraph of large minimum degree} Let $G=(V,E)$ be a graph. There exists $S\subseteq V$ such that $G[S]$ is a connected graph of minimum degree at least $|E|/|V|$.
\end{lemma}
The next lemma follows from a simple application of Hall's Theorem (see e.g. the exercises of Chapter 2 in \cite{Diestel}).
\begin{lemma}
\label{lemma: star matching}
Let $G=(V,E)$ be a bipartite graph with bipartition $V=A\cup B$, and let $k$ be a natural number. Suppose that for every $I\subseteq A$ we have $|N(I)|\ge k|I|$. Then for every $i\in A$ there exists a subset $J_i\subseteq N(i)$ of size $|J_i|=k$ such that all the sets $(J_i)_{i\in A}$ are disjoint.
\end{lemma}

A routine way to turn a non-Hamiltonian graph $H$ that satisfies
some expansion properties into a
Hamiltonian graph is by using {boosters}. A \emph{booster} is a
non-edge $e$ of $H$ such that the addition of $e$ to $H$ creates a path which is longer than a longest path of $H$, or
 turns $H$ into a Hamiltonian graph. In order to turn $H$ into a
Hamiltonian graph, we start by adding a booster $e$ of $H$. If the
new graph $H\cup \{e\}$ is not Hamiltonian then one can continue by
adding a booster of the new graph. Note that after at most $|V(H)|$
successive steps the process must terminate and we end up with a
Hamiltonian graph. The main point using this method is that it is
well-known (for example, see \cite{bollobas1998random}) that a
non-Hamiltonian connected graph $H$ with ``good'' expansion
properties has many boosters. In the proof of Theorem
\ref{thm:main1} we use a similar notion of boosters, known as
\emph{$e$-boosters}.

Given a graph $H$ and a pair $e\in \binom{V(H)}{2}$, consider a path $P$ of $H\cup \{e\}$ of maximal length which contains $e$ as
an edge. A non-edge $f$ of $H$ is called an $e$-booster if
$H\cup\{e,f\}$ contains a path $Q$ which passes through $e$ and
which is longer than $P$ or if $H\cup\{e,f\}$ contains a
Hamilton cycle that uses $e$. The following lemma appears in
\cite{FKN} and shows that every connected and non-Hamiltonian graph
$G$ satisfying certain expansion properties has many $e$-boosters for every possible
$e$.
\begin{lemma} \label{lem:boosters}
Let $H$ be a connected graph for which $|N_{H}(X)\setminus X|\ge
2|X|+2$ holds for every subset $X\subseteq V(H)$ of size $|X|\le k$.
Then, for every pair $e\in \binom{V(H)}{2}$ such that $H\cup\{e\}$
does not contain a Hamilton cycle which uses the edge $e$, the
number of $e$-boosters for $H$ is at least $\frac{1}{2}(k+1)^2$.
\end{lemma}

\section{Proof of Theorem \ref{thm:main1}}

In this section we prove Theorem \ref{thm:main1}. In our proof we
present a randomised algorithm which successively queries (about adjacency) carefully
selected pairs of vertices in $\mathcal G(n,p)$, where $p\ge \frac{\ln n+\ln\ln n+w(1)}{n}$. We then show that whp the algorithm terminates by revealing a Hamilton cycle after only $n+o(n)$ positive answers. The algorithm is divided
into five different phases, labeled I, II, III, IV and V.

We remark that we may and will assume throughout the proof that $p\le \frac{10\ln n}{n}$. Indeed, if $p>\frac{10\ln n}{n}$ then we can use the algorithm Alg($p'$) which queries pairs of vertices with probability $p'=\frac{10\ln n}{n}$ with a slight modification to obtain an algorithm Alg($p$) which queries pairs of vertices with probability $p$. When a pair of vertices is queried by Alg($p'$), do it in two stages: first query it with probability $\frac{p'}{p}$ to decide whether Alg($p$) should query this pair of vertices as well; if the answer is positive, then query it a second time with probability $p$. A pair of vertices which is queried by Alg($p'$) is considered to be an edge if and only if the answer to both questions is positive, and so this happens with probability $p'=\frac{p'}{p}\cdot p$. However, in the algorithm Alg($p$) pairs of vertices are queried about adjacency with probability $p$. Finally, note crucially that the edges which are revealed by the algorithm Alg($p$) are exactly the same as the ones which are revealed by the algorithm Alg($p'$) and so, if the latter whp finds a Hamilton cycle after only $n+o(n)$ positive answers then so does the former.

In order to simplify notation in the proof, we work in the following setting. Throughout the algorithm we maintain a tripartition $R\cup W\cup B$ of the edge set of the complete graph with vertex set $V=[n]$. Edges in $R$, $W$, $B$ are called respectively \textit{red}, \textit{white} and \textit{blue}. A red edge represents an edge which has been queried successfully (and thus belongs to the exposed graph $G$), a white edge represents an edge which has not yet been queried and a blue edge represents an edge which has been queried unsuccessfully. During the algorithm we \textit{recolour} some white edges. \textit{Recolouring} a white edge means that with probability $p$ we recolour it red (i.e., we move it from the set $W$ to the set $R$), and otherwise we recolour it blue (i.e. we move it from the set $W$ to the set $B$). All the recolourings are considered independent. At any point during the algorithm, the \textit{red graph} (respectively, \textit{white graph} and \textit{blue graph}) refers to the graph with vertex set $V$ and edge set $R$ (respectively, $W$ and $B$). Moreover, if $u,v\in V$ then we say that $v$ is a \textit{red neighbour} (respectively, \textit{white neighbour} and \textit{blue neighbour}) of $u$ if $uv\in R$ (respectively, $uv\in W$ and $uv\in B$). The algorithm starts with the tripartition
$(R,W,B)=\left(\emptyset,\binom{V}{2}, \emptyset\right)$ and whp it
ends with the red graph containing a Hamilton cycle while having
only $n+o(n)$ edges.

During the algorithm, if $R$ denotes the set of red edges at a certain point, and if at that point it is verified that any of the events below does not hold then we stop the algorithm:
\begin{enumerate}[N.1]
\item We have $\Delta(R)\le 40\ln n$.
\item If none of the edges incident to a given vertex $v\in V$ are white (i.e. they were already recoloured before) then $v$ has at least $2$ red neighbours.
\item There are no two cycles in $R$ of length at most $(\ln n)^{0.9}$ sharing exactly one vertex.
\item Between any two vertices $u,v\in V$ there are at most three trails of length at most $6$ in $R$.
\end{enumerate}
We remark that all of these events hold whp by $(P1)$ of Lemma \ref{lemma:propertiesofGnp} and by Lemma \ref{lemma: minimum degree 2 and no two small cycles are close}, and so we can assume throughout the proof that these properties always hold.

In Phases I--IV we consider a finite number of properties concerning the tripartition $(R,W,B)$, which we need for later phases. These properties will be labeled according to the phase in which they are considered, in order to make it easier for the reader to find them in a later reference. For example, II.1(b) will be used to denote property 1(b) of Phase II. In each phase, we show that all the properties considered hold whp and so we may assume that they hold for later phases.

\subsection{Outline of the algorithm}

{\bf Phase I:} In this phase we use a modified version of the well
known graph search algorithm Depth First Search (see e.g.
\cite{West}). Starting from a complete graph with all edges white,
we use this algorithm to find a ``long" red path $\mathcal P$ by
recolouring red at most $n-1$ white edges. Afterwards by recolouring
red one more white edge between ``short" initial and final segments
of the path $\mathcal P$, we create a red cycle $\mathcal C_1$ of
size $n-\Theta(n(\ln n)^{-0.45})$. This is done whilst ensuring some
technical conditions needed for later phases.

{\bf Phase II:} Let $\mathcal U:=V\setminus V(\mathcal{C}_1)$. In
this phase, starting from the partition $V=V(\mathcal{C}_1)\cup
\mathcal{U}$, we recolour a random subset of white edges in
$\mathcal U$ and partition $\mathcal{U}$ into three sets
$\text{EXP}_1$, $\text{SMALL}_1$ and $\text{TINY}$. The set
$\text{EXP}_1$ will be such that the minimum red degree inside it is
$\Omega\left((\ln n)^{0.4}\right)$ and $|\mathcal{U}\setminus
\text{EXP}_1|=o\left(n/\ln n\right)$. Later, we recolour all the
white edges between $\mathcal{U}\setminus \text{EXP}_1$ and
$V(\mathcal{C}_1)$. A partition $\mathcal U\setminus
\text{EXP}_1=\text{SMALL}_1\cup\text{TINY}$ is then obtained by
letting $\text{SMALL}_1$ be the set of all vertices of ``large" red
degree into a ``large" subset $\mathcal M_1$ of $V(\mathcal C_1)$.
Finally, we recolour all the white edges inside $\mathcal U$
touching vertices of $\text{TINY}$. All of this is achieved by
recolouring red only $o(n)$ edges during this phase.

{\bf Phase III:} The goal of this phase is to ``swallow" the
vertices of $\text{TINY}$ one at a time into the red cycle
$\mathcal{C}_1$ obtained in Phase I. This is achieved by creating at
each time a larger red cycle that contains a new vertex of
$\text{TINY}$ and some vertices of $\text{EXP}_1$, until a red cycle
$\mathcal{C}_2$ such that $V(\mathcal{C}_1)\cup \text{TINY}\subseteq
V(\mathcal{C}_2)$ is obtained. We ensure that whp only $o(n)$ edges
are recoloured red during this phase. At the end of this phase we
get a partition of the vertex set $V=V(\mathcal{C}_2)\cup
\text{EXP}_2\cup\text{SMALL}_2$ where $\text{EXP}_2\subseteq
\text{EXP}_1$ is a ``good expander" and $\text{SMALL}_2\subseteq
\text{SMALL}_1$ is a set of vertices of ``large" degree into a
``large" subset $\mathcal M_2\subseteq \mathcal M_1$ of $V(\mathcal
C_2)$.

{\bf Phase IV:} The goal of this phase is to ``swallow" the vertices of $\text{SMALL}_2$ one at a time into the red cycle $\mathcal{C}_2$ obtained in Phase III. This is achieved by creating at each time a larger red cycle that contains a new vertex of $\text{SMALL}_2$, until a red cycle $\mathcal{C}_3$ such that $V(\mathcal{C}_2)\cup \text{SMALL}_2\subseteq V(\mathcal{C}_3)$ is obtained. We ensure that whp only $o(n)$ edges are recoloured red during this phase. Moreover, at the end we get a partition of the vertex set $V=V(\mathcal{C}_3)\cup \text{EXP}_3$ where $\text{EXP}_3\subseteq \text{EXP}_2$ is a ``good expander".

{\bf Phase V:} In this phase we create a Hamilton cycle in the red
graph by merging the red cycle $\mathcal C_3$ with the set
$\text{EXP}_3$. We start by recolouring red $\Theta(1)$ white edges
in $\text{EXP}_3$ in order to make the red graph in $\text{EXP}_3$
become connected. Afterwards, we consider two adjacent vertices
$v,w$ in the red cycle $\mathcal C_3$ which have large white degree
onto $\text{EXP}_3$. By recolouring edges between $v,w$ and
$\text{EXP}_3$ we then find $x,y\in \text{EXP}_3$ such that $vx$ and
$wy$ are coloured red. Finally, using the fact that the red graph in
$\text{EXP}_3$ is a connected expander we recolour red at most
$|\text{EXP}_3|$ edges in order to find a red Hamilton path from $x$
to $y$ inside the set $\text{EXP}_3$. This path together with the
red path $\mathcal C_3\setminus\{vw\}$ and the red edges $vx$ and
$wy$ then provides the desired red Hamilton cycle in $V$. All of
this is achieved by recolouring red only $o(n)$ edges during this
phase.

\subsection{Phase I}
The algorithm for this phase is divided into two stages. In Stage 1
we use a randomised version of the Depth First Search exploration
algorithm to obtain a ``long" red path $\mathcal P$. In Stage 2 of
the algorithm we use the red path $\mathcal P$ to find a red cycle
$\mathcal{C}_1$ of size $n-\Theta(n(\ln n)^{-0.45})$, by recolouring
red exactly one white edge between an initial and a final interval
of $\mathcal P$.
\medskip

\textbf{(Stage 1)} In this stage we run a (slightly modified version of)
DFS algorithm on the vertex set $V=[n]$. Recall that DFS is an
algorithm to explore all the connected components of a graph
$G=(V,E)$, while finding a spanning tree of each of them in the
following way. It maintains a tripartition $(C,A,U)$ of the vertex
set $V$, letting $C$ be the set of vertices whose exploration is
complete, $U$ be the set of unvisited vertices and $A=V\setminus
(C\cup U)$ (the vertices which are ``active"), where the vertices of
$A$ are kept in a stack (last in first out). It starts with
$C=A=\emptyset$ and $U=V$ and runs until $A\cup U=\emptyset$. In
each round of the algorithm, if $A\neq \emptyset$, then it
identifies the last vertex $a\in A$, and starts to query $U$ for
neighbours of $a$, according to the natural ordering on them. If
such a neighbour exists, let $u\in U$ be the first such neighbour,
then the algorithm moves $u$ from $U$ to $A$. Otherwise, the
algorithm moves $a$ from $A$ to $C$. If $A=\emptyset$, then the
algorithm moves the first (according to the natural ordering) vertex
in $U$ to $A$.

The following properties of DFS will be relevant for us and follow
immediately from its description.

\begin{enumerate}[$(O1)$]

\item At any point during the algorithm, it is true that all the
pairs between $C$ and $U$ have been queried, and none of them are
edges of $G$.
\item Throughout the algorithm, the explored graph is a forest.
\item At each round of the algorithm exactly one vertex moves,
either from $U$ to $A$ or from $A$ to $C$.
\item The set $A$ always spans a path.
\end{enumerate}

For our purposes, we run DFS on a random input in the following way.
At each round of the algorithm let $e$ be the relevant pair waiting
to be queried. We first decide with probability $q=(\ln n)^{-1/2}$
whether we want to recolour this pair. If yes, we recolour the pair $e$ red
with probability $p$ (and consider it as an edge for DFS) and blue otherwise (and consider it as a non-edge for DFS). If no, we consider $e$ to be a non-edge for DFS. All these actions happen independently at
random. We stop this
algorithm as soon as $|C|=|U|$ (and not when $A\cup U=\emptyset$),
with the set $A$ spanning a red path $\mathcal P=a_1a_2\ldots
a_{|A|}$ of size $|A|=n-|C|-|U|$. Claim \ref{claim:StageIterminates} below ensures that whp at the end
of the routine one has $k:= n(\ln n)^{-0.45}\geq |C|=|U|$, so that
$|A|\geq n-2k$. We assume henceforth that this holds and we proceed
to Stage $2$.

\vspace{-10pt}
\begin{flushright}
\textbf{(End of Stage 1)}
\end{flushright}
\vspace{-4pt}

\textbf{(Stage 2)} Consider the intervals
$I_1:=\{a_1,a_2,\ldots,a_k\}$ and
$I_2:=\{a_{n-3k+1},a_{n-3k+2},\ldots, a_{n-2k}\}$ of $\mathcal P$
and let $D$ be the set of white edges between $I_1$ and
$I_2$. Following a fixed ordering of the set $D$, at each step first
decide with probability $q$ whether this edge should be recoloured.
If yes, then recolour it red with probability $p$ (and blue
otherwise). All these actions are taken independently. The stage is completed successfully when the first red edge from $D$ is discovered. Claim
\ref{claim:StageIterminates} below ensures that whp there will be such an edge. Assuming this, let $a_ia_j\in D$ be this edge,
where $a_i\in I_{1}$ and $a_j\in I_{2}$. The algorithm terminates by
setting $\mathcal{C}_1$ to be the red cycle formed by the vertices
$a_i,a_{i+1},\ldots, a_j$ with the red edges of $\mathcal P$
together with the red edge $a_ia_j$.

\vspace{-10pt}
\begin{flushright}
\textbf{(End of Stage 2)}
\end{flushright}
\vspace{-4pt}

In the next claim we prove that some properties which are assumed in
the algorithm hold whp.
\begin{claim}
  \label{claim:StageIterminates}
The following properties hold whp:
\begin{enumerate} [$(i)$]
\item At the end of Stage 1 we have $|C|=|U|\leq
k:=n(\ln n)^{-0.45}$.
\item During Stage 2, at least one edge in $D$ is recoloured red.
\end{enumerate}
\end{claim}

\begin{proof}[Proof of Claim \ref{claim:StageIterminates}]
In order to prove $(i)$, note that during the algorithm, each pair which has been queried
has been recoloured red with probability $pq\geq n(\ln n)^{-1/2}$,
independently at random. Moreover, it follows from $(O1)$ that none of the pairs between $C$ and $U$ have been coloured red. Note
that by exploring all the edges of the graph we obtain a graph which
is distributed as $\mathcal G(n,pq)$. Now, since
\[pq\geq \frac{(\ln n)^{0.5}}{n}\ge \frac{1.35(\ln n)^{0.45}\ln\ln n}{n}= \frac{3\ln \left(n/k\right)}{k}\]
it follows from Lemma \ref{pseudorandom} that this graph is whp
$k$-pseudorandom and therefore, unless $|C|=|U|\leq k$, there must
be red edges between these sets.

Property $(ii)$ follows from a similar argument.
\end{proof}

Assuming that the properties of Claim \ref{claim:StageIterminates} hold, denoting by $R_1$,
$W_1$ and $B_1$ the sets of red, white and blue edges, respectively, at the end of
this phase's algorithm, we show that whp the following technical
conditions hold:
\begin{enumerate}
\item[I.1] the graph $R_1$ contains a cycle $\mathcal{C}_1$ of size $t$, where $2n(\ln n)^{-0.45}< n-t\leq 4n(\ln n)^{-0.45}$.
\item[I.2] at most $n$ white edges are recoloured red during this phase, i.e. $|R_1|\le n$.
\item[I.3] for every $v\in V$ we have $d_{W_1}(v,V(\mathcal{C}_1))\ge n-5n(\ln n)^{-0.45}=(1-o(1))n$.
\item[I.4] letting $\mathcal{U}:=V\setminus V(\mathcal{C}_1)$, we have $d_{R_{1}\cup
B_{1}}(v,\mathcal U)\leq 4n(\ln n)^{-0.5}=o\left(|\mathcal
U|\right)$ for every $v\in V$.
\end{enumerate}

\begin{claim}
  \label{claim:PropertiesIhold}
  Properties I.1-I.4 hold (whp).
\end{claim}
\begin{proof}[Proof of Claim \ref{claim:PropertiesIhold}]
Note that the red cycle $\mathcal{C}_1$ formed by the vertices
$a_i,a_{i+1},\ldots, a_{j}$ obtained at the end of the algorithm has
size $t:=j-i+1$. Moreover, assuming Claim
\ref{claim:StageIterminates}, since $1\le i\le k$ and $n-3k+1\le
j\le n-2k$ we obtain that $n-4k<n-4k+2\le j-i+1=t\le n-2k$. This
settles I.1.

Since a forest in a graph of order $n$ has less than $n$ edges,
it is clear by (O2) that in Stage 1 of the algorithm less than $n$
edges are recoloured red. Moreover, since in Stage 2 only one
edge is recoloured red, it follows that in the whole phase at
most $n$ edges are recoloured red, settling I.2.

Next, note that $R_1\cup B_1$ can be seen as part of a graph
distributed as $\mathcal{G}(n,q)$. Thus, by (P1) of Lemma
\ref{lemma:propertiesofGnp} it follows that I.4 holds whp.
Furthermore, by I.1 and I.4 we have that for every $v\in V$
\begin{align*} d_{W_1}(v,V(\mathcal{C}_1))&\ge|V(\mathcal{C}_1)|-1-d_{R_1\cup
B_1}(v,V(\mathcal{C}_1))\\
&\ge n-4n(\ln n)^{-0.45}-1-\Delta(R_1\cup B_1)\ge n-5n(\ln
n)^{-0.45}\,
\end{align*}
which settles I.3.
\end{proof}
We have shown that whp at the end of this phase all the properties I.1-I.4 hold. We shall assume henceforth that they hold for the sets $R_1$, $W_1$ and $B_1$ obtained after this Phase.

\subsection{Phase II}
In this phase we partition $\mathcal{U}$ into three sets
$\mathcal{U}=\text{EXP}_1\cup\text{SMALL}_1\cup\text{TINY}$ as
described in the outline. The algorithm for this phase is divided
into the following three stages.
\medskip

\textbf{(Stage 1)} Let $F$ be a subset of $W_1[\mathcal{U}]$ obtained by independently adding each edge in $W_1[\mathcal{U}]$ to $F$ with probability $q':=
6(\ln n)^{-0.15}$. Claim \ref{claim:StageIIterminates} below ensures that whp $\frac{2}{3}|\mathcal U|q'\le \delta(F)\le \Delta(F)\le \frac 43 |\mathcal U|q'$. Assuming this, recolour all the edges in $F$.

Taking the set $F$ at random serves two purposes. Firstly, it
ensures that not too many edges are recoloured red in this phase.
Secondly, it leaves a certain amount of randomness for the edges in
$W_1[\mathcal U]\setminus F$, which will be used in later phases.
\vspace{-14pt}
\begin{flushright}
\textbf{(End of Stage 1)}
\end{flushright}
\vspace{-4pt}

\textbf{(Stage 2)}
Let $R^1$ denote the set of red edges after Stage $1$ and set $T_0:=\{v\in \mathcal U: d_{R^1\setminus R_1}(v,\mathcal U)<\frac{1}{3}|\mathcal{U}|pq'\}$. Claim \ref{claim:StageIIterminates} ensures that whp $|T_0|\le ne^{-(\ln n)^{0.4}}$. Assuming this, starting with $i=0$, as long as there exists a vertex $v\in \mathcal{U}\setminus T_i$ having at least $3$ red neighbours in $T_i$, choose such a vertex $v$ and set $T_{i+1}:=T_i\cup \{v\}$. Let $T_{f}$ be the last set obtained in this process. Claim \ref{claim:StageIIterminates} below shows that whp $f\le |T_0|$. Assuming that, define $\text{EXP}_1:= \mathcal{U}\setminus T_f$.

Note that every vertex $v\in
\text{EXP}_1$ has at most two red neighbours in $T_f$. Thus,  by I.1 for every $v\in \text{EXP}_1$ we have $d_{R^{1}\setminus R_1}(v,\text{EXP}_1)\ge \frac 13 |\mathcal U|pq'
-2\ge 3(\ln n)^{0.4}$.
\vspace{-14pt}
\begin{flushright}
\textbf{(End of Stage 2)}
\end{flushright}
\vspace{-4pt}

\textbf{(Stage 3)} Let $P_1,\ldots, P_m$ be vertex disjoint subpaths
of the red cycle $\mathcal{C}_1$, each of size $100$, where
$m=\left\lfloor \frac{|V(\mathcal{C}_1)|}{100}\right\rfloor$, and
set $\mathcal M_1$ to be the union of all the vertices in the
subpaths $P_1,\ldots, P_m$ which are not endpoints. These paths will
be used in later phases for technical reasons to ensure that certain
vertices are not neighbours on the red cycle $\mathcal C_1$.

Next, recolour all the white edges between $T_f$ and
$V(\mathcal{C}_1)$, set $\text{SMALL}_1$ to be the set of all
vertices in $T_f$ with at least $(\ln n)^{0.5}$ red neighbours in
$\mathcal M_1$ and set $\text{TINY}:=T_f\setminus\text{SMALL}_1$.
The algorithm for this phase ends by recolouring all the edges in
$W_1[\mathcal U]\setminus F$ touching at least one vertex in
$\text{TINY}$. \vspace{-6pt}
\begin{flushright}
\textbf{(End of Stage 3)}
\end{flushright}
\vspace{-4pt}

In the next claim we prove that some properties which are assumed in
the algorithm hold whp.

\begin{claim}
  \label{claim:StageIIterminates}
All of the following properties hold whp:
\begin{enumerate}[$(i)$]
\item In Stage 1 one has $\frac 23 |\mathcal{U}|q'\le \delta(F)\le\Delta(F)\le\frac 43
|\mathcal{U}|q'$.
\item In Stage 2 one has $|T_0|\le ne^{-(\ln n)^{0.4}}$ and $f\le |T_0|$.
\end{enumerate}
\end{claim}

\begin{proof}
First we prove that $(i)$ holds whp. For estimating
$\delta(F)$, note that by I.4 we have $d_{W_1}(v,\mathcal
U)=(1-o(1))|\mathcal U|$ for every $v\in \mathcal U$. Thus, by
Chernoff's inequality we conclude that for a vertex $v\in \mathcal
U$ we have $$\Pr\left[d_F(v)<\frac 23 |\mathcal U|q'\right]\leq
e^{-\Theta(|\mathcal U|q')}=e^{-\Theta(n(\ln n)^{-0.6})}.$$ Now, by
applying the union bound over all vertices of $v\in\mathcal U$, it follows that whp $\delta(F)\geq
\frac 23 |\mathcal U|q'$. In a similar way, we can also conclude
that whp $\Delta(F)\le \frac 43|\mathcal U|q'$. This settles $(i)$.

Assuming $(i)$, we show next that whp $|T_0|\le
ne^{-(\ln n)^{0.4}}$. Indeed, by Chernoff's inequality we see that for every $v\in \mathcal{U}$:
\[\Pr\left[v\in T_0\right]= \Pr\left[d_{R^1\setminus R_1}(v,\mathcal{U})< \frac 13 |\mathcal{U}|pq'\right]\leq \Pr\left[\text{Bin}\left(\frac 23|\mathcal{U}|q',p\right)\leq
\frac 13|\mathcal{U}|pq'\right]\leq e^{-(\ln n)^{0.4}},\]
where in the last inequality we used the fact that
$|\mathcal{U}|pq'\ge 12(\ln n)^{0.4}$ by I.1. Therefore, the expected
value of $|T_0|$ is at most $|\mathcal{U}|e^{-(\ln
n)^{0.4}}$. Hence, using Markov's inequality we obtain that whp
$|T_0|\le ne^{-(\ln n)^{0.4}}$, as desired.

Finally, we show that if $|T_0|\le ne^{-(\ln n)^{0.4}}$
then whp $f\le |T_0|$. Suppose that $f>|T_0|$. Then, note that
the set $T_{|T_0|}$ contains precisely $2|T_0|$ vertices and induces
at least $3|T_0|=1.5|T_{|T_0|}|$ red edges, since for every $i\le |T_0|$ the vertex in
$T_{i}\setminus T_{i-1}$ has at least 3 red neighbours in $T_{i-1}$. By (P2) of Lemma \ref{lemma:propertiesofGnp} (with $c= 1.5$), since $p\le \frac{10\ln n}{n}$, we know that whp every subset of vertices $X$ of size
\[|X| \le \left(\frac{1}{\ln n}\cdot \frac{3^{1.5}}{e^{2.5}np^{1.5}}\right)^{2} \le \frac{3^{3}}{e^{5}\cdot 10^{3}}\cdot \frac{n}{(\ln n)^5}\]
 induces less than $1.5|X|$ red edges. Since $|T_{|T_0|}| = 2|T_{0}|\le 2ne^{-(\ln n)^{0.4}} = o(n(\ln n)^{-5})$ it follows that whp $f\le |T_{0}|$. Thus, we conclude that whp $(ii)$ holds as claimed.
\end{proof}

Assuming that the properties of Claim \ref{claim:StageIIterminates} hold, denoting by $R_2$, $W_2$ and $B_2$ the sets of red, white
and blue edges at the end of this phase, we show that
whp the following technical conditions hold:

\begin{enumerate}
\item[II.1] Properties of the set $\text{EXP}_1$:
\begin{enumerate}
\item $|\text{EXP}_1|\ge\left(1-\frac{1}{(\ln n)^3}\right)|\mathcal U|\ge (2-o(1))n(\ln n)^{-0.45}$.
\item for every $v\in \text{EXP}_1$ we have $d_{R_2\setminus R_1}(v,\text{EXP}_1)\ge 3(\ln n)^{0.4}$.
\item for every set $U\subseteq \text{EXP}_1$  of size $|U|\ge \left(1-\frac{1}{\ln n}\right)|\text{EXP}_1|$:
\begin{enumerate}
\item if $S\subseteq U$ is a set such that $(R_2\setminus R_1)[S]$ has minimum degree at least $(\ln n)^{0.4}$ then for any set $X\subseteq S$ of size $|X|\le \frac{1}{6000}n(\ln n)^{-0.45}$ we have $|N_{R_2[S]}(X)\cup X|\ge 5|X|$.
\item there is a set $S\subseteq U$ of size $|S|\ge \frac{1}{240}n(\ln n)^{-0.45}$ such that $R_2[S]$ has diameter at most $2\ln n$.
\end{enumerate}
\item Let $\mathcal H$ be the collection of all connected graphs $H$ whose vertex set $V(H)$ is a subset of $\text{EXP}_1$ and whose edge set is of the form $K_1\cup K_2$ where $K_1=R_2[V(H)]$ is such that $|N_{K_1}(X)\cup X|\ge 5|X|\text{ for every } X\subseteq V(H) \text{ of size }|X| \le \frac{1}{6000}n(\ln n)^{-0.45}$ and $K_2\subseteq \binom{V(H)}{2}$ is a set of size $|K_2|\le |V(H)|+24000$. Then whp for every $H\in \mathcal H$ and every $e=xy\in \binom{V(H)}{2}$, if the graph $H\cup \{e\}$ does not contain a Hamilton cycle which uses the edge $e$, then the number of $e$-boosters for $H$ in the set $W_{2}[V(H)]$ is at least $10^{-8}n^2(\ln n)^{-0.9}$.
\end{enumerate}
\item[II.2] Properties of the set $\text{SMALL}_1$:
\begin{enumerate}
\item $|\text{SMALL}_1|\leq 2ne^{-(\ln n)^{0.4}}$.
\item $d_{R_2}(u,\mathcal M_1)\ge (\ln n)^{0.5}$ for every vertex $u\in \text{SMALL}_1$.
\end{enumerate}
\item[II.3] Properties of the set $\text{TINY}$:
\begin{enumerate}
\item $|\text{TINY}|\le n^{0.04}$.
\item the event ``there is no red path of size at most $1000$ between any two vertices of $\text{TINY}$ after recolouring all the edges in $W_2$" holds whp.
\item $d_{R_2}(u)\ge 2$ for every vertex $u\in \text{TINY}$.
\end{enumerate}
\item[II.4] Only $o(n)$ edges in $W_1$ are recoloured red during this phase.
\end{enumerate}

It is clear that properties II.1(a), II.1(b), II.2(a) and II.2(b) follow immediately from the algorithm. Moreover, note that at the end of this phase all the edges touching vertices in $\text{TINY}$ are either red or blue. Thus, since the event N.2 holds whp it follows that property II.3(c) also holds whp. The next few claims ensure that the remaining properties all hold whp.

\begin{claim}
  \label{claim:PropertyII.1(c)holds}
  Property II.1(c) holds whp.
\end{claim}

\begin{proof}[Proof of Claim \ref{claim:PropertyII.1(c)holds}]
Consider the following two events:
\begin{enumerate}[(E1)]
\item $e_{R_2\setminus R_1}(S)<C(\ln n)^{0.4}|S|$ for any set $S\subseteq \text{EXP}_1$ of size $|S|\le\frac{C}{120}n(\ln n)^{-0.45}$, for $C\in\left\{\frac{1}{10},\frac{1}{2}\right\}$.
\item $e_{R_2}(X,Y)>0$ for every pair of disjoint sets $X,Y\subseteq \mathcal{U}$ each of size at least $\frac{1}{6000} n(\ln n)^{-0.45}$.
\end{enumerate}
We shall prove that these two events all hold whp and then that II.1(c) holds whenever N.1, II.1(b), (E1) and (E2) all hold.

The event (E1) holds whp according to (P3) of Lemma
\ref{lemma:propertiesofGnp} (with $C$ in the Lemma being $C(\ln n)^{0.4}$ here) since in Stage $1$ every edge in $W_1[\mathcal U]$ is recoloured red independently with probability $pq'\le \frac{60(\ln n)^{0.85}}{n}$. Next we show that the event (E2) also holds whp. By I.4, after Phase I we have
$e_{W_1}(X,Y)\geq (1-o(1))s^2$ for any pair of disjoint sets
$X,Y\subseteq \mathcal{U}$ each of size at least $s:=\frac{1}{6000}n(\ln
n)^{-0.45}$. Thus, by applying Chernoff's bound and the union
bound we obtain
\vspace{-2mm}
\begin{align*}
&\Pr\left[\exists \text{ such sets } X,Y \text{ with }
 e_{R_2}(X,Y)=0
\right]
\le \binom{n}{s}^2(1-pq')^{(1-o(1))s^2}\\
&\le \left(\frac{en}{s}\right)^{2s}e^{-(1-o(1))pq's^2}\le e^{O\left(n(\ln n)^{-0.45}\ln \ln n\right)}
e^{-\Omega\left(n(\ln n)^{-0.05}\right)}=o(1).
\end{align*}
implying that (E2) holds whp as desired.

Suppose now that N.1, II.1(b), (E1) and (E2) all hold. Let
$U\subseteq \text{EXP}_1$ be a subset of size $|U|\ge
\left(1-\frac{1}{\ln n}\right)|\text{EXP}_1|$ and denote $R_2\setminus R_1$ simply by $R'$.

Suppose $S\subseteq U$ is such that $R'[S]$ has minimum degree at least $(\ln
n)^{0.4}$ and that there exists a set $X\subseteq S$
of size $|X|\le \frac{1}{6000}n(\ln n)^{-0.45}$ such that $|N_{R'[S]}(X)\cup
X|< 5|X|$. Since $R'[S]$ has minimum degree at least $(\ln
n)^{0.4}$ it follows that
\[e_{R'}(N_{R'[S]}(X)\cup X)\geq \frac{1}{2}(\ln n)^{0.4}|X|\ge
\frac{1}{10}(\ln n)^{0.4}|N_{R'[S]}(X)\cup X|.\] Thus, by (E1) we have
$|N_{R'[S]}(X)\cup X|\geq \frac{1}{1200}n(\ln n)^{-0.45}$, which leads to
$|X|>\frac{1}{6000}n(\ln n)^{-0.45}$, contradicting our choice of $X$. We conclude that for every set $X\subseteq S$ of size $|X|\le \frac{1}{6000}n(\ln n)^{-0.45}$ we have $|N_{R_2[S]}(X)\cup X|\ge |N_{R'[S]}(X)\cup X|\ge 5|X|$. This settles i.\ of II.1(c).

Observe now that by N.1 and II.1(b) we have
\begin{align*}
e_{R'}(U)&\ge e_{R'}(\text{EXP}_1)-\Delta(R')\cdot |\text{EXP}_1\setminus U|\ge \frac{3}{2}(\ln n)^{0.4}|\text{EXP}_1|-40\ln n\cdot \frac{|\text{EXP}_1|}{\ln n}\ge (\ln n)^{0.4}|U|.
\end{align*}
Thus, by Lemma \ref{lemma:subgraph of large minimum degree} there
exists a set $S\subseteq U$ such that $R'[S]$ is a connected graph
with minimum degree at least $(\ln n)^{0.4}$. In particular, we have
\[e_{R'}(S)\ge \frac{1}{2}(\ln n)^{0.4}|S|\]
and so by (E1) it follows that $|S|\ge \frac{1}{240}n(\ln n)^{-0.45}$.

For $z\in S$, set $N^{0}(z):=\{z\}$ and, for $i\ge 1$, define
$N^{i}(z):=N_{R_2[S]}(N^{i-1}(z))\cup N^{i-1}(z)$. Note crucially that every vertex
in $N^{i}(z)$ is at distance at most $i$ of $z$ in $R_2[S]$ and
that, by the above, we have $|N^{i}(z)|\ge 5|N^{i-1}(z)|$,
provided $|N^{i-1}(z)|\le \frac{1}{6000}n(\ln n)^{-0.45}$. Thus, if we
take $\ell:=\log_5\left(\frac{1}{6000}n(\ln n)^{-0.45}\right)$, we
see that $|N^{\ell}(z)|\ge \frac{1}{6000}n(\ln n)^{-0.45}$ for any
$z\in S$. Now, let $x,y\in S$ be distinct. Note that if
$N^{\ell}(x)\cap N^{\ell}(y)\neq \emptyset$ then $x$ is at distance
at most $2\ell$ from $y$ in $R_2[S]$. If instead we have
$N^{\ell}(x)\cap N^{\ell}(y)=\emptyset$ then by (E2) there is at
least one edge in $R_2$ between $N^{\ell}(x)$ and $N^{\ell}(y)$,
implying that $x$ and $y$ are at distance at most $2\ell+1$ in
$R_2[S]$. Since $2\ell+1<2\ln n$, this settles ii.\ of II.1(c). We
conclude that II.1(c) holds whp as claimed.
\end{proof}

\begin{claim}\label{GnpContainsManyBoosters}
Property II.1(d) holds whp.
\end{claim}

\begin{proof} Note first that by Lemma \ref{lem:boosters}, given such an
$H\in \mathcal H$  and $e\in \binom{V(H)}{2}$, the number of $e$-boosters for $H$ in $\binom{V(H)}{2}$ is at least $\frac{1}{72\cdot 10^6}n^2(\ln
n)^{-0.9}$. Moreover, since $e$-boosters of $H$ are not edges of $H$ and since all the edges in $R_2[V(H)]$ are edges of $H$, it follows that every $e$-booster for $H$ is either in $B_2[V(H)]$ or in $W_2[V(H)]$. Note that by properties I.4 and I.1 we have \[e_{B_1}(V(H))\le \frac{1}{2}\cdot 4n(\ln n)^{-0.5}\cdot |\mathcal U|\le 8n^2(\ln n)^{-0.95}=o\left(n^2(\ln n)^{-0.9}\right).\] Furthermore, observe that for every $e$-booster of $H$ which is not in $B_1[V(H)]$, the probability that it is in $W_2[V(H)]$ is at least the probability that it does not belong to $F$ and so at least $1-q'=1-o(1)$. Moreover, it is clear that the latter events are independent. Thus, the probability that less than $10^{-8}n^2(\ln n)^{-0.9}$ $e$-boosters for $H$ are in the set $W_2[V(H)]$ is at most
\[\Pr\left[\text{Bin}\left((1-o(1))\frac{1}{72\cdot 10^6}n^2(\ln
n)^{-0.9}, 1-o(1)\right)< (1-o(1))10^{-8}n^2(\ln n)^{-0.9}\right]\le e^{-\Theta(n^2(\ln n)^{-0.9})}\]
by Chernoff's bound (Lemma \ref{Che}).

Suppose now that $|R_2|\le 20n\ln n$. Assuming this, it is clear that any graph in $\mathcal H$ has at most $20n\ln n+n$ edges and so \[|\mathcal H|\le \sum_{i=0}^{20n\ln n+n}\binom{n^2}{i}\le (20n\ln n+n+1)\binom{n^2}{20n\ln n+n}\le e^{100n(\ln n)^2}.\] Thus, using the union bound we see that the probability that, for some $H\in \mathcal H$ and some $e=xy\in \binom{V(H)}{2}$ such that the graph $H\cup \{e\}$ does not contain a Hamilton cycle which uses the edge $e$, the number of $e$-boosters for $H$ in the set $W_2[V(H)]$ is less than $10^{-8}n^2(\ln n)^{-0.9}$, conditioning on the fact that $|R_2|\le 20n\ln n$, is at most
\[e^{100n(\ln n)^2}n^2e^{-\Theta(n^2(\ln
n)^{-0.9})}=o(1).\]
Since whp we have $|R_2|\le 20n\ln n$ according to N.1, the claim follows.
\end{proof}

\begin{claim}
  \label{claim:PropertyII.3holds}
  Properties II.3(a) and II.3(b) hold whp.
\end{claim}

\begin{proof}[Proof of Claim \ref{claim:PropertyII.3holds}]
Note, by the definition of the set $\mathcal M_1$, that $|\mathcal
M_1|\geq 0.98|V(\mathcal C_1)|-100$ and recall that by
I.3, at the end of Phase I, for every $v\in V$ we have
$d_{W_1}(v,V(\mathcal{C}_1))\geq n-5n(\ln n)^{-0.45}$. Therefore,
for every $v\in V$ we have (say) $d_{W_1}(v,\mathcal M_1)\geq
0.97n$. Note that for a vertex $v\in T_f$ we have $d_{R_2\setminus
R_1}(v,\mathcal M_1)\sim \Bin(d_{W_1}(v,\mathcal M_1),p)$. Thus,
since $\ln n\le np\le 10\ln n$, it follows that for any vertex $v\in
T_f$: \vspace{-5mm}
\begin{align*}
\Pr[v\in \text{TINY}] \le\Pr[\text{Bin}(0.97n,p)\le(\ln n)^{0.5}]\le\sum_{i=0}^{(\ln n)^{0.5}}\binom{n}{i}p^i(1-p)^{0.97n-i}\\
\le (1-p)^{0.97n}\sum_{i=0}^{(\ln n)^{0.5}}\left(\frac{e np}{i
(1-p)}\right)^i\le e^{-0.97\ln n}((\ln n)^{0.5}+1)(2enp)^{(\ln
n)^{0.5}} \le n^{-0.96}\,,
\end{align*}
and therefore $\mathbb{E}\left[|\text{TINY}|\right]\leq
|T_f|n^{-0.96}\leq 2n^{0.04}e^{-(\ln n)^{0.4}}$ (recall that whp
$|T_f|\leq 2ne^{-(\ln n)^{0.4}}$). Therefore, we conclude by
Markov's inequality that whp $|\text{TINY}|\leq n^{0.04}$, settling
II.3(a).

Now we show that II.3(b) also holds whp. Let $R^2$ and $W^2$ denote the sets of red and white edges after Stage $2$. We assume throughout that $\Delta(R^2)\le 40\ln n$ (which holds whp by N.1). Suppose that in Stage $3$ instead of only recolouring all the edges in $W^{2}$ between $T_f$ and $V(\mathcal{C}_1)$ we had decided to recolour all the edges in $W^{2}$. Let $R$ denote the set of red edges after this recolouring process.  We would like to stress at this point that the edges in $R^2\subseteq R$ are considered to be fixed and only the edges in $W^2$ are regarded as being randomly and independently assigned to $R$ with probability $p$.

Let $\mathcal {P}$ be the set of all paths in $K_n$ of size at most $1000$. For each $P\in \mathcal{P}$ consider the indicator random variable $X_P$ of the event that $P$ is a path in the graph formed by the edges in $R$. Finally, let $X=\sum_{P\in\mathcal{P}}X_{P}$ denote the total number of paths in $\mathcal{P}$ which are paths in $R$. Note that for each $v\in V$, the number of paths in $\mathcal{P}$ starting with $v$ which are also paths in $R$ is at most $\Delta(R)+\Delta(R)^2+\ldots +\Delta(R)^{1000}\leq
1000\cdot\Delta(R)^{1000}$. Thus, it is clear that $X\le 1000n\cdot\Delta(R)^{1000}$ and so we have
\[\mathbb{E}\left[X|\Delta(R)\le (\ln n)^2 \right]\le 1000n(\ln n)^{2000}\le n^{1.1}.\]
Moreover, note that as $\Delta(R^2)\le 40 \ln n$ we have by Lemma \ref{Che2}:
\begin{align*}
&\Pr\left[\Delta(R)> (\ln n)^2\right]\le \Pr\left[\Delta(R\setminus R^2)> 0.5(\ln n)^2\right]\le \sum_{v\in V}\Pr\left[d_{R\setminus R^2}(v)>0.5(\ln n)^2\right]\\
&\le n\Pr\left[\text{Bin}(n,p)>0.5(\ln n)^2\right]\le n \left(\frac{10e\ln n}{0.5(\ln n)^2}\right)^{0.5(\ln n)^2}\le e^{-(\ln n)^2}.
\end{align*}
For every pair of vertices $\{u,v\}\subseteq T_f$ let $\mathcal{P}_{u,v}\subseteq \mathcal{P}$ be the collection of paths in $K_n$ of size at most $1000$ with endpoints $u$ and $v$. For every $u,v\in T_f$ and $P\in \mathcal{P}_{u,v}$ consider now the families
\begin{align*}\mathcal{A}_{u,v}:=\{A\subseteq W^2: u,v\in\text{TINY}\text{ if } R=R^2\cup A\}\text{ and }\mathcal{B}_{P}:=\{B\subseteq W^2:E(P)\subseteq R  \text{ if }R=R^2\cup B\}.
\end{align*}
Observe that the families $\mathcal{A}_{u,v}$ and $\mathcal{B}_{P}$ are monotone decreasing and monotone increasing in the universe $W^2$, respectively. Furthermore, note crucially that the event ``$u,v\in \text{TINY}$" is exactly the event ``$R\setminus R^2\in \mathcal{A}_{u,v}$" and that the event ``$E(P)\subseteq R$" is exactly the event ``$R\setminus R^2\in \mathcal{B}_P$". Since each edge in $W^2$ is in $R\setminus R^2$ independently with probability $p$ it follows from Theorem $6.3.2$ of \cite{AlonSpencer} that for every $u,v\in T_f$ and $P\in \mathcal{P}_{u,v}$ we have:
\[\Pr\left[u,v\in \text{TINY and }E(P)\subseteq R\right]\le \Pr\left[u,v\in \text{TINY}\right]\cdot\Pr\left[E(P)\subseteq R\right].\]
Thus, using the union bound and the estimates above, the probability that there exist $u,v\in \text{TINY}$ and $P\in \mathcal{P}_{u,v}$ for which $E(P)\subseteq R$ is at most
\begin{align*}
&\sum_{\{u,v\}\subseteq T_f}\sum_{P\in \mathcal{P}_{u,v}}\Pr\left[u,v\in \text{TINY and }E(P)\subseteq R\right]\le \sum_{\{u,v\}\subseteq T_f}\sum_{P\in \mathcal{P}_{u,v}} \Pr\left[u,v\in \text{TINY}\right]\cdot\Pr\left[E(P)\subseteq R\right]\\
\le&\sum_{\{u,v\}\subseteq T_f}\sum_{P\in \mathcal{P}_{u,v}}\left(n^{-0.96}\right)^2\cdot\left(\Pr\left[E(P)\subseteq R|\Delta(R)\le (\ln n)^2\right]+\Pr\left[\Delta(R)> (\ln n)^2\right]\right)\\
\le & \; n^{-1.92}\sum_{\{u,v\}\subseteq T_f}\sum_{P\in \mathcal{P}_{u,v}}\left(\mathbb{E}\left[X_P|\Delta(R)\le (\ln n)^2\right]+e^{-(\ln n)^2}\right)\\
\le &\; n^{-1.92}\left(\mathbb{E}\left[X|\Delta(R)\le (\ln
n)^2\right]+e^{-(\ln n)^2}|\mathcal P|\right)\le
n^{-1.92}\left(n^{1.1}+e^{-(\ln n)^2}1000n^{1000}\right) =o(1)\,,
\end{align*}
where in the second inequality we used the fact that the events $u\in \text{TINY}$ and $v\in \text{TINY}$ are independent. This settles II.3(b).\end{proof}

\begin{claim}
  \label{claim:PropertyII.4holds}
  Property II.4 holds whp.
\end{claim}
\begin{proof}[Proof of Claim \ref{claim:PropertyII.4holds}]
In this phase edges are recoloured once in Stage $1$ and in two instances in Stage $3$. We shall bound the number of edges recoloured red in these three instances.

In Stage $1$ we recoloured all the edges of $F$, of which there are at most $\frac{1}{2}\Delta(F)\cdot |\mathcal U|$. Since by the algorithm we have whp that $\Delta(F)\le \frac{4}{3}|\mathcal U|q'$ we conclude using Chernoff's bounds and I.1 that whp the number of edges recoloured red in Stage $1$ is at most
\[\Delta(F)\cdot |\mathcal U|\cdot p\le \frac{4}{3}|\mathcal U|q'\cdot|\mathcal U|p=O\left(\frac{n}{(\ln n)^{0.05}}\right)=o(n)\ .\]

In Stage $3$ we recoloured all the white edges between $T_f$ and $V(\mathcal{C}_1)$ and all the white edges touching vertices of $\text{TINY}$. Since, from the algorithm, $|T_f|\le 2|T_0|\le  2ne^{-(\ln n)^{0.4}}$, we conclude from Chernoff's bounds that whp the number of edges between $T_f$ and $V(\mathcal{C}_1)$ which are recoloured red is at most
\[2|T_f|\cdot |V(\mathcal C_1)|\cdot p\le 4ne^{-(\ln n)^{0.4}}\cdot n \cdot p=O\left(\frac{n\ln n}{e^{(\ln n)^{0.4}}}\right)=o(n)\ .\]
Finally, since there are at most $|\text{TINY}|n$ edges touching vertices of $\text{TINY}$, we can use Chernoff's bounds together with II.3(a) to conclude that whp the number of edges touching vertices of $\text{TINY}$ which are recoloured red is at most
\[2|\text{TINY}|\cdot n\cdot p=O\left(n^{0.04}\ln n\right)=o(n)\ .\]
Thus, whp $o(n)$ edges in $W_1$ are recoloured red in this phase, proving the claim.
\end{proof}

We have shown that whp at the end of this phase all of the properties II.1--II.4 hold. We shall assume henceforth that all these properties hold for the sets $R_2$, $W_2$ and $B_2$.

\subsection{Phase III}
In this phase, we want to find a red cycle $\mathcal C_2$ containing
$\text{TINY}\cup V(\mathcal C_1)$ as described in the outline.
Recall that in Stage 3 of Phase II we nearly decomposed the red cycle
$\mathcal C_1$ into
$m=\left\lfloor\frac{|V(\mathcal{C}_1)|}{100}\right\rfloor$ red
paths $P_1,\ldots,P_m$, needed for technical reasons for later
phases. We ensure in this phase that the red cycle $\mathcal C_2$
will be such that most of these paths are also paths in
$\mathcal{C}_2$. Concretely, we obtain a set $J\subseteq [m]$ of
size $|J|\ge (1-o(1))m$ such that all the paths $(P_j)_{j\in J}$ are
paths in the red cycle $\mathcal{C}_2$. At the end of this phase we
get a partition of the vertex set $V=V(\mathcal{C}_2)\cup
\text{EXP}_2\cup\text{SMALL}_2$ where $\text{EXP}_2\subseteq
\text{EXP}_1$ is a ``good expander" and $\text{SMALL}_2\subseteq
\text{SMALL}_1$ is such that every vertex $v\in \text{SMALL}_2$ has
``large" red degree onto the set $\mathcal M_2$ which is the union
of all the vertices of the paths $(P_j)_{j\in J}$ which are not
endpoints.

The algorithm for this phase is divided into $t=|\text{TINY}|$
parts. For each $i\in [t]$ we define during Part $i$ the following
sets:
\begin{itemize}
\item $R^{i}$ and $W^{i}$ which denote, respectively, the sets of all edges which are coloured red and white at the end of Part $i$.
\item $\text{EXP}^{i}\subseteq \text{EXP}_1$, $\text{SMALL}^{i}\subseteq \text{SMALL}_1$ and $\mathcal U^{i}\subseteq \mathcal{U}$, the latter being the union of $\text{EXP}^{i}$, $\text{SMALL}^{i}$ and $t-i$ vertices of $\text{TINY}$.
\end{itemize}
During Part $i$, we recolour ``some" edges in $W^{i-1}$ in order to
obtain a red cycle $\mathcal{C}^{i}$ (i.e. consisting solely of
edges in $R^{i}$) such that $V(\mathcal C^{i})=V\setminus \mathcal
U^{i}$ contains $V(\mathcal{C}_1)$, $i$ vertices from $\text{TINY}$
and ``few" vertices from $\mathcal U$. During this phase's algorithm
we keep track of which red paths $(P_j)_{j\in [m]}$ are also paths
in the red cycle $\mathcal{C}_2$. To this end, we use the function
$j:V\rightarrow \{0,1,\ldots, m\}$ defined as:
\[j(v):=\left\{\begin{matrix}
j & \text{if } v\in V(P_j) \text{ for } j\in [m]\\
0 & \text{if } v\notin \bigcup_{j=1}^{m}V(P_j)
\end{matrix}\right..\]
During Part $i$ we maintain a set $J^{i}\subseteq [m]$ such that for
every $j\in J^{i}$ the path $P_j$ is a path in the cycle $\mathcal
C^{i}$. The algorithm for this phase is as follows:

\medskip

{\bf Algorithm:} Fix an enumeration $x_1,x_2,\ldots, x_t$ of the
vertices in $\text{TINY}$, where $t=|\text{TINY}|$, and set
$\mathcal{C}^{0}:=\mathcal{C}_1$, $\mathcal U^{0}:=\mathcal{U}$,
$\text{EXP}^{0}:= \text{EXP}_1$, $\text{SMALL}^{0}:=\text{SMALL}_1$,
$R^{0}:=R_2$, $W^{0}:=W_2$ and $J^{0}:=[m]$. For $i=1,2,\ldots,t$
execute the following routine which shows how to add $x_i$ to the
red cycle $\mathcal C^{i-1}$:

{\bf Routine:} Recall from II.3(c) that $d_{R_2}(x_i)\ge 2$. Thus,
since $R_2\subseteq R^{i-1}$, exactly one of the following holds:
\begin{enumerate}[(a)]
\item $d_{R^{i-1}}(x_i,V(\mathcal C^{i-1}))\geq 1$ and $d_{R^{i-1}}(x_i,\mathcal U^{i-1})\ge 1$.
\item $d_{R^{i-1}}(x_i, V(\mathcal C^{i-1}))\ge 2$ and $d_{R^{i-1}}(x_i,\mathcal U^{i-1})=0$.
\item $d_{R^{i-1}}(x_i, V(\mathcal C^{i-1}))=0$ and $d_{R^{i-1}}(x_i,\mathcal U^{i-1})\ge 2$.
\end{enumerate}
We proceed depending on which of the cases above holds. For each of these cases we consider two red neighbours of $x_i$ and depending on whether they lie in $V(\mathcal C^{i-1})$ or in $\mathcal  U^{i-1}$ we use them in a certain way to incorporate $x_i$ into the red cycle $\mathcal C^{i-1}$. For the
sake of simplicity, we only describe here how to proceed if (a) holds. However, we stress that this case contains all the ideas necessary for treating the other two cases. Essentially, for case (b) (resp. (c)) the two red neighbours of $x_i$ considered should be treated as the red neighbour of $x_i$ in case (a) which lies in $V(\mathcal C^{i-1})$ (resp. $\mathcal U^{i-1}$). If case (a) holds, proceed as follows:

Fix a cyclic enumeration $v_1v_2\ldots v_{\ell}$ of
the vertices in the red cycle $\mathcal C^{i-1}$, where
$\ell=|V(\mathcal C^{i-1})|$ (indices considered modulo $\ell$), and let $z^{i}_1\in V(\mathcal C^{i-1})$ and $z^{i}_2\in \mathcal U^{i-1}$ be two red neighbours of $x_i$. Without loss of generality we assume that $z^{i}_1=v_{\ell}$.

Set $\text{BAD}_{i}:=\{v\in
\text{EXP}^{i-1}:d_{R_2}(v,\text{EXP}^{i-1})<3(\ln n)^{0.4}\}$ to
be the set of all vertices in $\text{EXP}^{i-1}$ with ``low" red
degree inside $\text{EXP}^{i-1}$ and set
$\text{GOOD}_{i}:=\text{EXP}^{i-1}\setminus(\text{BAD}_{i}\cup
N_{R_2}(\text{BAD}_{i})\cup \{z^{i}_2\})$. Let $S_{i}\subseteq
\text{GOOD}_{i}$ be a subset of size at least $\frac{1}{240}n(\ln n)^{-0.45}$
such that $R^{i-1}[S_{i}]$ is a connected graph of diameter at most
$2\ln n$. Claim \ref{claim:PhaseIIIterminates} below ensures that whp such a set $S_i$ exists and we assume this henceforth.

For each $v_j\in V(\mathcal C^{i-1})$ define $s^{+}(v_j):=v_{j+1}$ and $s^{-}(v_j):=v_{j-1}$ to be the ``successor" and ``predecessor" of $v_j$ in the cycle $V(\mathcal C^{i-1})$ (notice that $s^{+}(z_1^{i}) = s^{+}(v_{\ell})=v_1$). Recolour all the edges in $W^{i-1}$ between $\{s^{+}(z^{i}_{1}),z^{i}_{2}\}$ and $V(\mathcal{C}_1)$ and, letting $R^{i}_1$ and $W^{i}_1$ denote, respectively, the sets of red and white edges at this point, consider the sets $A^{i}_1:=N_{R^{i}_1}(s^{+}(z^{i}_1),V(\mathcal C_1))\setminus \{z^{i}_1\}$ and $A^{i}_2:=N_{R^{i}_1}(z^{i}_2,V(\mathcal C_1))\setminus\{z^{i}_1,s^{+}(z^{i}_1)\}$. Note that for any two vertices $v_{a}\in A^{i}_1$ and $v_{b}\in
A^{i}_2$ we have a red path
\[P(v_a,v_b):=\left\{\begin{matrix}
v_{a-1}v_{a-2}\ldots v_2 v_1 v_av_{a+1}\dots v_{b-1}v_b z^{i}_2 x_i v_l v_{l-1}\ldots v_{b+1} & \text{if } 1<a\le b<l\\
v_{a-1}v_{a-2}\ldots v_{b+1}v_b z^{i}_2 x_i v_l v_{l-1}\ldots v_{a+1}v_a v_1 v_2 \ldots v_{b-1} & \text{if } 1<b<a<l
\end{matrix}\right.\]
from $s^{-}(v_a)=v_{a-1}$ to either $s^{+}(v_b)=v_{b+1}$ or $s^{-}(v_b)=v_{b-1}$ such that
$V(P(v_a,v_b))=V(\mathcal C^{i-1})\cup \{x_i,z^{i}_2\}$. Define $B^{i}_1:=\{s^{-}(v):v\in A^{i}_1\}$ to be the set of possible initial vertices of these paths.

Recolour all the edges in $W^{i}_1$ between vertices in $B^{i}_1$ and $S_{i}$ and, letting $R^{i}_2$ and $W^{i}_2$ denote, respectively, the sets of red and white edges at this point, let $y^{i}_1\in B^{i}_1$ and $u^{i}_1\in S_{i}$ be such that $y^{i}_1u^{i}_1\in R^{i}_2$. Claim \ref{claim:PhaseIIIterminates} ensures that whp such vertices exist and we assume this henceforth. Define now $B^{i}_2$ to be the set of possible final vertices of the paths $P(s^{+}(y^{i}_1),v)$ where $v\in A^{i}_2$.

Recolour all the edges in $W^{i}_2$ between vertices in $B^{i}_2$ and $S_{i}$ and, letting $R^{i}$ and $W^{i}$ denote, respectively, the sets of red and white edges at this point, let $y^{i}_2\in B^{i}_2$ and $u^{i}_2\in S_{i}$ be such that $y^{i}_2u^{i}_2\in R^{i}$. Claim \ref{claim:PhaseIIIterminates} ensures that whp such vertices exist and we assume this henceforth. Moreover, let $s(y^{i}_2)\in \{s^{+}(y^{i}_2),s^{-}(y^{i}_2)\}$ be the vertex of $V(\mathcal C^{i-1})$ such that $y^{i}_2$ is the final vertex of the red path $P(s^{+}(y^{i}_1),s(y^{i}_2))$.

Let $P(u^{i}_1,u^{i}_2)$ be a path inside $S_{i}$ from $u^{i}_1$ to
$u^{i}_2$ of length at most $2\ln n$ consisting solely of edges in
$R^{i}$ (such a path exists by the choice of $S_{i}$) and set
$\mathcal C^{i}$ to be the red cycle formed by joining the red paths
$P(s^{+}(y^{i}_1),s(y^{i}_2))$ and $P(u^{i}_1,u^{i}_2)$ with the red
edges $y^{i}_1u^{i}_1$ and $y^{i}_2u^{i}_2$. Furthermore, set
$J^{i}:=J^{i-1}\setminus (\{j(z^{i}_1)\}\cup
\{j(y^{i}_1)\}\cup\{j(y^{i}_2)\})$ to be the set of indices
obtained by deleting the indices of the paths we ``broke" during
this routine, and note that every path $P_j$ with $j\in J^i$ is still
a subpath of the red cycle $\mathcal C^i$ (provided it was also a
subpath of $\mathcal C^{i-1}$). Finally, set
$\text{EXP}^{i}:=\text{EXP}^{i-1}\setminus V(\mathcal C^{i})$,
$\text{SMALL}^{i}:=\text{SMALL}^{i-1}\setminus V(\mathcal{C}^{i})$
and $\mathcal{U}^{i}:=\mathcal{U}^{i-1}\setminus V(\mathcal C^{i})$.
\vspace{-14pt}
\begin{flushright}
\textbf{(End of Routine)}
\end{flushright}
\vspace{-4pt}
To end the algorithm set $\text{EXP}_2:=\text{EXP}^{t}$,
$\text{SMALL}_2:=\text{SMALL}^{t}$, $\mathcal{C}_2:=\mathcal C^{t}$,
$J:=J^{t}$ and $\mathcal M_2$ to be the union of all the inner
vertices of the paths $(P_j)_{j\in J}$. \vspace{-14pt}
\begin{flushright}
\textbf{(End of Algorithm)}
\end{flushright}
\vspace{-4pt}

We make a few observations about the procedure above which will be important for later:
\begin{enumerate}[(O1)]
\item For every $i\in [t]$ we have $|\text{EXP}^{i-1}\setminus \text{EXP}^{i}|\le 2\ln n + 3$. This is because in Part $i$ we have $\text{EXP}^{i-1}\setminus \text{EXP}^{i}\subseteq V(P(u^{i}_1,u^{i}_2))\cup \{z^{i}_1,z^{i}_2\}$ where $P(u^{i}_1,u^{i}_2)$ is a path of size at most $2\ln n +1$ (we might need to remove not just $z^{i}_2$ but also $z^{i}_1$ from $\text{EXP}^{i-1}$ if case (c) in the algorithm holds). Moreover, since by II.1(b) we have $d_{R_2}(v,\text{EXP}_1)\ge 3(\ln n)^{0.4}$ for every $v\in \text{EXP}_1$, every vertex $v\in \text{BAD}_{i}$ must have at least one red neighbour in $\text{EXP}_1\setminus \text{EXP}^{i-1}$. Thus, we have $|\text{BAD}_{i}|\le |\text{EXP}_1\setminus \text{EXP}^{i-1}|\cdot \Delta(R_2)\le (2\ln n+3)(i-1)\cdot \Delta(R_2)$.

\item For every $i\in [t]$ and every $j\in \{1,2\}$ we have $d_{R^{i}}(v^{i}_{j},V(\mathcal{C}_1))-2\le |B^{i}_j|\le d_{R^{i}}(v^{i}_{j},V(\mathcal{C}_1))$ for some vertex $v^{i}_{j}$ which is at distance at most $2$ from $x_i$ in $R_2$. For example, if case (a) holds in Part $i$ then we have $v^{i}_{1}=s^{+}(z^{i}_1)$ and $v^{i}_{2}=z^{i}_2$. Moreover, for every $i\in [t]$ and every vertex $v\in B^{i}_1\cup B^{i}_2$ there is a path in $R^{i}$ of length at most $4$ from $x_i$ to $v$. This follows immediately from the definition of the sets $B^{i}_{j}$.

\item We have $N_{W^{i-1}}(v,\text{EXP}^{i-1})=N_{W_1}(v,\text{EXP}^{i-1})$ for every $v\in V(\mathcal{C}_1)\setminus \left(\bigcup_{j=1}^{i-1}\left(B^{j}_1\cup B^{j}_2\right)\right)$ and every $i\in [t+1]$. This is because between Phase I and Part $i$ of this phase's algorithm the only edges that are recoloured between $V(\mathcal{C}_1)$ and $\text{EXP}^{i-1}$ touch vertices of $\bigcup_{j=1}^{i-1}\left(B^{j}_1\cup B^{j}_2\right)$.
\end{enumerate}
Also, for the rest of this phase, we shall assume that the following event occurs:
\begin{enumerate}[(E1)]
\item for every $i\in [t]$ there is no path of size at most $1000$ consisting solely of edges in $R^{i}$ between any two vertices in $\text{TINY}$.
\end{enumerate}
Note that this event occurs whp as indicated in II.3(b). In the next claim we prove that some properties which are assumed in the algorithm hold whp.

\begin{claim}\label{claim:PhaseIIIterminates}
All of the following properties hold whp:
\begin{enumerate}[$(i)$]
\item For any $i\in[t]$ there always exists a set $S_i\subseteq \text{GOOD}_{i}$ of size at least $\frac{1}{240}n(\ln n)^{-0.45}$ such that $R^{i-1}[S_i]$ is a connected graph of diameter at most $2\ln n$.
\item For any $i\in [t]$, in Part $i$, after recolouring all the edges in $W^{i-1}$ between the sets $S_i$ and $B^{i}_1\cup B^{i}_2$, there exist $y^{i}_1\in B^{i}_1$, $y^{i}_2\in B^i_2$ and $u^{i}_1,u^{i}_2\in S_i$ such that
$y^{i}_ju^{i}_j\in R^i$ for $j\in \{1,2\}$.
\end{enumerate}
\end{claim}

\begin{proof}[Proof of Claim \ref{claim:PhaseIIIterminates}]
We start by proving that $(i)$ holds whp. Assuming that $\Delta(R_2)\le 40\ln n$ (which holds whp after Phase II, by N.1), we have by (O1), II.1(a) and II.3(a) that
\begin{align*}
|\text{GOOD}_{i}|&\ge |\text{EXP}^{i-1}|-|\text{BAD}_{i}\cup N_{R_2}(\text{BAD}_{i})|-2\\
&\ge |\text{EXP}_1|-(2\ln n+3)\cdot(i-1)-|\text{BAD}_{i}|\cdot (1+40\ln n)-2\ge \left(1-\frac{1}{\ln n}\right)|\text{EXP}_1|.
\end{align*}
We remark that the $-2$ after the first inequality is necessary if case (c) holds (for case (a) one only needs $-1$). Thus, by II.1(c) there is always a set $S_{i}\subseteq \text{GOOD}_{i}$ with the desired properties.

Next we show that $(ii)$ holds whp. By (O2) we know that for every $i\in [t]$ and every $j\in \{1,2\}$ we have $|B^{i}_j|\ge d_{R^i}(v^{i}_{j},V(\mathcal{C}_1))-2$ for some vertex $v^{i}_{j}$ which is at distance at most $2$ from $x_i$ in $R_2$. We claim that whp $d_{R^i}(v^{i}_{j},V(\mathcal{C}_1))\ge (\ln n)^{0.5}$ for every $i\in [t]$ and $j\in \{1,2\}$.

Note that $v^{i}_{j}\notin \text{TINY}$ as otherwise $v^{i}_{j}$ and $x_i$ would be two vertices in $\text{TINY}$ at distance at most $2$ in $R^i$, contradicting (E1). Moreover, recall from II.2(b) that if $v^{i}_{j}\in \text{SMALL}_1$ then $d_{R^i}(v,V(\mathcal{C}_1))\ge (\ln n)^{0.5}$. If $v^{i}_{j}\in \text{EXP}_1\cup V(\mathcal{C}_1)$ then recall from I.3 that $d_{W_1}(v^{i}_j,V(\mathcal{C}_1))=(1-o(1))n$ and note that $d_{W_2}(v^{i}_j,V(\mathcal{C}_1))=d_{W_1}(v^{i}_j,V(\mathcal{C}_1))$ as no edges between the sets $\text{EXP}_1\cup V(\mathcal{C}_1)$ and $V(\mathcal{C}_1)$ were recoloured during Phase II. Thus, using the union bound, Lemma \ref{Che} and II.3(a) we see that the probability that $d_{R^{i}}(v^{i}_j,V(\mathcal{C}_1))< (\ln n)^{0.5}$ for some $i\in [t]$ and $j\in\{1,2\}$ is at most
\[2t\cdot \Pr\left[\text{Bin}((1-o(1))n,p)< (\ln n)^{0.5}\right]\le 2n^{0.04}\cdot e^{-\left(\frac{1}{2}-o(1)\right)\ln n}=o(1)\,,\]
as claimed. Hence, whp $|B^{i}_j|\ge (\ln n)^{0.5}-2$ for every $i\in [t]$ and $j\in \{1,2\}$. We assume this hereafter.

Note that for $i\neq i'$ we have $(B^{i}_1\cup B^{i}_2)\cap (B^{i'}_{1}\cup B^{i'}_2)=\emptyset$ since otherwise (O2) would imply that there is a path in $R^{i}$ of length at most $8$ between $x_i$ and $x_{i'}$, contradicting (E1). Thus, by (O3) and I.4, we see that $d_{W^{i-1}}(v,S_{i})=d_{W_1}(v,S_{i})=(1-o(1))|S_i|$ for every $v\in B^{i}_{1}\cup B^{i}_{2}$ and every $i\in [t]$.

Now, for each $i\in [t]$ and $j\in\{1,2\}$, let $C^{i}_j\subseteq B^{i}_j$ be a subset of size at least $\left(\frac{1}{2}-o(1)\right)|B^{i}_j|\ge \frac{1}{3}(\ln n)^{0.5}$ such that $C^{i}_{1}\cap C^{i}_{2}=\emptyset$.
We then see that for any $i\in [t]$ and $j\in\{1,2\}$ the probability that there is no edge in $R^i$ between $C^{i}_j$ and $S_{i}$ is at most:
\[\Pr\left[\text{Bin}\left((1-o(1))|S_{i}|\cdot |C^{i}_{j}|,p\right)=0\right]= (1-p)^{(1-o(1))|S_{i}|\cdot |C^{i}_{j}|}\le e^{-\frac{\ln n}{n}\cdot \frac{1-o(1)}{240} n(\ln n)^{-0.45}\cdot \frac{1}{3}(\ln n)^{0.5}}\le \frac{1}{n}.\]
Using II.3(a) and the union bound we see that the probability that for some $i\in [t]$ and $j\in\{1,2\}$ there is no edge in $R^{i}$ between $B^{i}_{j}$ and $S_{i}$  is at most $2t\cdot \frac{1}{n} = o(1)$. This shows that $(ii)$ holds whp.
\end{proof}

Assuming that the properties of Claim \ref{claim:PhaseIIIterminates} hold, denoting by $R_3$, $W_3$ and $B_3$ the sets of red, white and blue edges at the end of this phase's algorithm, we show that whp the following technical conditions hold:
\begin{enumerate}
\item[III.1] Properties of $\text{EXP}_2$:
\begin{enumerate}
\item $|\text{EXP}_2|\ge\left(1-\frac{1}{(\ln n)^2}\right)|\mathcal U|\ge (2-o(1))n(\ln n)^{-0.45}$.
\item for every $v\in \text{EXP}_2$ we have $d_{R_3\setminus R_1}(v,\text{EXP}_2)\ge 2(\ln n)^{0.4}$.
\end{enumerate}
\item[III.2] Properties of $\text{SMALL}_2$:
\begin{enumerate}
\item $|\text{SMALL}_2|\le |\text{SMALL}_1|\le 2ne^{-(\ln n)^{0.4}}$.
\item for every $v\in \text{SMALL}_2$ we have $d_{R_3}(v,\mathcal M_2)\ge (\ln n)^{0.5}-400$.
\item for every $v\in \text{SMALL}_2$ we have $N_{W_3}(u,\text{EXP}_2)\neq N_{W_1}(u,\text{EXP}_2)$ for  at most 100 vertices $u\in V(\mathcal C_1)$ which are at distance at most $2$ from $v$ in $R_3$.
\end{enumerate}
\item[III.3] All the paths $(P_j)_{j\in J}$ are paths in the red cycle $\mathcal C_2$.
\item[III.4] In this phase only $o(n)$ edges of $W_2$ are recoloured red.
\end{enumerate}

Note that property III.2(a) is just a consequence of the fact that $\text{SMALL}_2\subseteq \text{SMALL}_1$ together with II.2(a). Moreover, property III.3 follows immediately from the algorithm. Indeed, $E(\mathcal C^{i-1})\setminus E(\mathcal C^{i})$ consists of at most $4$ edges (only $3$ edges if case (a) holds but $4$ if case (b) holds) and, from the definition of $J^{i}$, for each such edge $e$ we remove $j(v)$ from $J^{i-1}$ for one vertex $v\in e$. Thus, it is clear that $P_j$ is still a path in the red cycle $\mathcal C^{i}$ for every $j\in J^{i}$. The next few claims show that the remaining properties all hold whp, assuming that the properties of Claim \ref{claim:PhaseIIIterminates} hold.

\begin{claim}
\label{claim:PropertiesIII.1hold}
Properties III.1(a) and III.1(b) hold whp.
\end{claim}
\begin{proof}[Proof of Claim \ref{claim:PropertiesIII.1hold}]
It follows from (O1), I.1, II.1(a) and II.3(a) that:
\[|\text{EXP}_2|=|\text{EXP}_1|-\sum_{i=1}^{t}|\text{EXP}^{i-1}\setminus \text{EXP}^{i}|\ge \left(1-\frac{1}{(\ln n)^3}\right)|\mathcal U|-t\cdot(2\ln n+3)\ge \left(1-\frac{1}{(\ln n)^2}\right)|\mathcal U|\]
and so this shows that III.1(a) holds.

We now show that III.1(b) also holds whp. Suppose $d_{R_3\setminus R_1}(v,\text{EXP}_2)<2(\ln n)^{0.4}$ for some $v\in \text{EXP}_2$ and let $i\in [t]$ be the largest possible integer such that $d_{R_2\setminus R_1}(v,\text{EXP}^{i-1})\ge 3(\ln n)^{0.4}$ (this is well defined by II.1(b)). Note that since $R_3[\mathcal U]=R_2[\mathcal U]$ we have
\[d_{R_3\setminus R_1}(v,\text{EXP}_2)=d_{R_2\setminus R_1}(v,\text{EXP}^{i-1})-\sum_{j=i}^{t}d_{R_2\setminus R_1}(v,\text{EXP}^{j-1}\setminus \text{EXP}^{j}).\]
Recall that $\text{EXP}^{j-1}\setminus \text{EXP}^{j}\subseteq V(P(u^{j}_1,u^{j}_2))\cup \{z^{j}_1,z^{j}_2\}\subseteq \text{GOOD}_j\cup \{z^{j}_1,z^{j}_2\}$ where $z^{j}_1$ and $z^{j}_2$ are neighbours of $x_j$ in $R_2$. Moreover, note that by the choice of $i$ we have $v\in \text{BAD}_{j}$ for every $j>i$. Thus, since $\text{GOOD}_{j}\cap N_{R_2}(\text{BAD}_j)=\emptyset$ we get that
\[d_{R_3\setminus R_1}(v,\text{EXP}_2)\ge d_{R_2\setminus R_1}(v,\text{EXP}^{i-1})-d_{R_2}(v,V(P(u^{i}_1,u^{i}_2)))-\sum_{j=i}^{t}d_{R_2}(v,\{z^{j}_1,z^{j}_2\}).\]
Observe now that if $d_{R_2}(v,\{z^{j}_1,z^{j}_2\})>0$ and $d_{R_2}(v,\{z^{j'}_1,z^{j'}_2\})>0$ for $j\neq j'$ then there would exist a path in $R_3$ of length at most $4$ between $x_j$ and $x_{j'}$, contradicting (E1). Thus, we can conclude that
\[d_{R_3\setminus R_1}(v,\text{EXP}_2)\ge d_{R_2\setminus R_1}(v,\text{EXP}^{i-1})-d_{R_2}(v,V(P(u^{i}_1,u^{i}_2)))-2\ge 3(\ln n)^{0.4}-d_{R_2}(v,V(P(u^{i}_1,u^{i}_2)))-2.\]
Since we assumed that $d_{R_3\setminus R_1}(v,\text{EXP}_2)<2(\ln n)^{0.4}$, we must have that
\[d_{R_2}(v,V(P(u^{i}_1,u^{i}_2)))>(\ln n)^{0.4}-2.\]
It is easy to see that this implies the existence of two cycles of length $O((\ln n)^{0.6})$ sharing only the vertex $v$ in the graph $R_2$ since the red path $P(u^{i}_1,u^{i}_2)$ has length at most $2\ln n$. However, if N.3 holds then this does not happen. Since N.3 holds whp we conclude that III.1(b) holds whp as desired.
\end{proof}
\begin{claim}
\label{claim:PropertiesIII.2hold}
Properties III.2(b) and III.2(c) hold whp.
\end{claim}

\begin{proof}[Proof of Claim \ref{claim:PropertiesIII.2hold}]
First we show that whp III.2(b) holds. Note that for every $v\in \text{SMALL}_2$ we have
\[d_{R_3}(v,\mathcal M_2)\ge d_{R_3}(v,\mathcal M_1)- \sum_{i=1}^{t}\sum_{j\in J^{i-1}\setminus J^{i}}d_{R_3}(v,V(P_j))\]
since $\mathcal M_1\setminus \left(\bigcup_{i=1}^{t}\bigcup_{j\in J^{i-1}\setminus J^{i}}V(P_j)\right)\subseteq \mathcal M_2$. Also, since $|J^{i-1}\setminus J^{i}|\le 4$ and since $|V(P_j)|=100$ we have that for every $i\in [t]$
\[\sum_{j\in J^{i-1}\setminus J^{i}}d_{R_3}(v,V(P_j))\le 400.\]
Observe now that if for $i\neq i'$, $j\in J^{i-1}\setminus J^{i}$ and $j'\in J^{i'-1}\setminus J^{i'}$ we have $d_{R_3}(v,V(P_j))>0$ and $d_{R_3}(v,V(P_{j'}))>0$ then there is a path in $R_3$ of size at most $2\cdot 102+2+1=207$ between $x_{i}$ and $x_{i'}$. Indeed, this follows from the fact that every vertex in $V(P_j)$ is at distance at most $3+(|V(P_j)|-1)=102$ of $x_i$ in $R^{i}$ and similarly that every vertex in $V(P_{j'})$ is at distance at most $102$ of $x_{i'}$ in $R^{i'}$. However, if the event (E1) holds then there does not exist such a path. Since (E1) holds whp we conclude that whp
\[d_{R_3}(v,\mathcal M_2)\ge d_{R_3}(v,\mathcal M_1)-400,\]
and so by II.2(b) we see that III.2(b) holds.

Now we show that whp III.2(c) holds. Note that by (O3) it is enough
to show that for every $v\in \text{SMALL}_2$ there are at most $100$
vertices $u\in\bigcup_{j=1}^{t}(B^{j}_1\cup B^{j}_2)$ which are at
distance at most $2$ from $v$ in $R_3$. Assume this is not the case
for a vertex $v\in \text{SMALL}_2$. Note first, if for some $j\neq
j'$ the sets $B^{j}_1\cup B^{j}_2$ and $B^{j'}_1\cup B^{j'}_2$ each
contain a vertex which is at distance at most $2$ from $v$ in $R_3$,
then by (O2) it follows that there is a red path from $x_{j}$ to
$x_{j'}$ of length at most $4+2+2+4=12$ in $R_3$. However, this
cannot occur if the event (E1) holds. Since (E1) holds whp we
conclude that whp for every $v\in \text{SMALL}_2$ there is at most
one $j\in [t]$ for which $B^{j}_1\cup B^{j}_2$ contains a vertex
which is at distance at most $2$ from $v$ in $R_3$. Moreover, if for
some $j\in [t]$ the set $B^{j}_1\cup B^{j}_2$ contains at least
$100$ vertices which are at distance at most $2$ from $v$ in $R_3$
then by (O2) one can find at least four trails of length at most
$4+2=6$ between $v$ and $x_j$. However, this cannot occur if N.4
holds. Since N.4 holds whp we conclude that whp III.2(c) holds.
\end{proof}

\begin{claim}
\label{claim:PropertyIII.4holds}
Property III.4 holds whp.
\end{claim}
\begin{proof}[Proof of Claim \ref{claim:PropertyIII.4holds}]
Recall that in Part $i$ of the algorithm we recolour edges in two situations. In the first situation we recolour edges between two vertices (if case (a) holds then these vertices are $s^{+}(z^{i}_1)$ and $z^{i}_2$) and $V(\mathcal{C}_1)$. In the second situation we recolour edges between vertices in $B^{i}_1\cup B^{i}_2$ and $S_{i}$. From the algorithm it is easy to see that the number of edges in the first situation which are recoloured red is at most $|B^{i}_{1}|+|B^{i}_{2}|+4$. For each $i\in [t]$ and $j\in \{1,2\}$ let $E^{i,j}_1$ be the event that $|B^{i}_j|> (\ln n)^{1.1}$ and let $E^{i,j}_2$ be the event that $e_{R^{i}\setminus R^{i-1}}(B^{i}_{j},S_{i})> 80(\ln n)^{1.65}$. Notice that if none of these events hold then clearly at most $(2(\ln n)^{1.1}+4)\cdot t + 80(\ln n)^{1.65}\cdot 2t=o(n)$ edges are recoloured red in this phase. Note that $\bigwedge_{i\in[t],j\in\{1,2\}}\overline{E^{i,j}_1}$ holds whp by (O2) and N.1. Moreover, using the fact that $S_i\subseteq \mathcal U$ and that $|\mathcal U|\le 4n(\ln n)^{-0.45}$ by I.1, we have by Chernoff (Lemma \ref{Che}) that
\begin{align*}
\Pr\left[E^{i,j}_2\wedge \overline{E^{i,j}_1} \right] \le &\Pr\left[E^{i,j}_2 \;|\; \overline{E^{i,j}_1}\right]\le \Pr\left[\text{Bin}\left((\ln n)^{1.1}\cdot 4n(\ln n)^{-0.45},p\right)> 80(\ln n)^{1.65}\right]\\
\le &\Pr\left[\text{Bin}\left(4n(\ln n)^{0.65},\frac{10\ln n}{n}\right)> 80(\ln n)^{1.65}\right]
\le e^{-\frac{40}{3}(\ln n)^{1.65}}.
\end{align*}
Thus, using the union bound and II.3(a), we see that the probability that the number of edges recoloured red during this phase is larger than $(2(\ln n)^{1.1}+4)\cdot t + 80(\ln n)^{1.65}\cdot 2t$ is at most
\[\Pr\left[\bigvee_{i\in[t],j\in\{1,2\}}E^{i,j}_1\right]+ \sum_{i\in [t],j\in\{1,2\}}\Pr\left[E^{i,j}_2\wedge \overline{E^{i,j}_1}\right]\le o(1)+2t\cdot  e^{-\frac{40}{3}(\ln n)^{1.65}}=o(1)\]
This, together with II.3(a), shows that whp III.4 holds.
\end{proof}

We have shown that whp at the end of this phase all of the properties III.1-III.4 hold. We shall assume henceforth that all these properties hold for the sets $R_3$, $W_3$ and $B_3$.

\subsection{Phase IV}
In this phase, we want to recolour some edges in $W_3$ in order to
find a red cycle $\mathcal C_3$ containing $\text{SMALL}_2\cup
V(\mathcal C_2)$ in a such a way that $\text{EXP}_3=V\setminus
V(\mathcal C_3)$ is a ``good expander" as described in the outline.
The algorithm for this phase is similar in spirit to the one of
Phase III. It is divided into three stages. In Stage 1 we define
notation and sets that will be useful for us throughout the
algorithm. Stage 2 is the main stage of the algorithm and is divided
into $s=|\text{SMALL}_2|$ parts. For each $i\in [s]$ we denote by
$R^{i}$ and $W^{i}$, respectively, the sets of all edges which are
coloured red and white at the end of Part $i$. During Part $i$, we
recolour ``some" edges in $W^{i-1}$ in order to obtain a red cycle
$\mathcal{C}^{i}$ (i.e. consisting solely of edges in $R^{i}$) such
that $V(\mathcal C^{i})$ contains $V(\mathcal{C}_2)$, $i$ vertices
from $\text{SMALL}_2$ and the vertex set of a ``small" red path in
$\text{EXP}_2$. Finally, in Stage $3$ we define sets needed for
later phases. The algorithm for this phase is as follows:
\medskip

{\bf Algorithm:} Fix an enumeration $y_1,y_2,\ldots, y_s$ of the vertices in $\text{SMALL}_2$, where $s=|\text{SMALL}_2|$, set $\mathcal{C}^{0}:=\mathcal{C}_2$, $\text{EXP}^{0}:= \text{EXP}_2$, $R^{0}:=R_3$ and $W^{0}:=W_3$ and for each $j\in J$ let $M_j$ denote the set of inner vertices of the path $P_j$.

Let $(J_i)_{i\in [s]}$ be disjoint subsets of $J$ of size $10^3(\ln n)^{0.45}$ such that $d_{R_3}(y_i,M_j)>0$ for every $j\in J_i$. Claim \ref{claim:PhaseIVterminates} ensures that such sets exist and we assume that henceforth.  For each $j\in J_i$ let $m_j\in M_j$ be such that $y_im_j\in R_3$. For $i=1,2,\ldots,s$ execute the following routine:

{\bf Routine:} Fix a cyclic orientation of the vertices in the red cycle $\mathcal C^{i-1}$ and denote for each $v\in V(\mathcal C^{i-1})$ by $s^{+}(v)$ and $s^{-}(v)$ the successor and predecessor of $v$ in the cycle $\mathcal C^{i-1}$ according to this orientation, respectively.

Set $\text{BAD}_{i}:=\{v\in \text{EXP}^{i-1}:d_{R_3}(v,\text{EXP}^{i-1})<2(\ln n)^{0.4}\}$, $\text{GOOD}_{i}:=\text{EXP}^{i-1}\setminus(\text{BAD}_{i}\cup N_{R_3}(\text{BAD}_{i}))$ and let $S_{i}\subseteq \text{GOOD}_{i}$ be a subset of size at least $\frac{1}{240}n(\ln n)^{-0.45}$ such that $R^{i-1}[S_{i}]$
is a connected graph of diameter at most $2\ln n$. Claim \ref{claim:PhaseIVterminates} ensures that whp such a set $S_i$ always exists and we assume that henceforth.

Define $A_i:= \{s^{+}(m_j):j\in J_i\}$ and recolour all the edges in $W^{i-1}$ between the set $A_i$ and
$S_{i}$. Letting $R^{i}$ and $W^{i}$ to be respectively the sets of red and white edges at this point, let $u^{i}_1,u^{i}_2\in S_i$ and $j^{i}_1,j^{i}_2\in J_i$ be distinct indices such that $s^{+}(m_{j^{i}_1})u^{i}_1\in R^{i}$ and $s^{+}(m_{j^{i}_2})u^{i}_2\in R^{i}$. Claim \ref{claim:PhaseIVterminates} ensures that whp such vertices $u^{i}_1,u^{i}_2$ and distinct indices $j^{i}_1,j^{i}_2$ always exist and we assume that hereafter.

Set
\[Q^i_1:=s^{+}(m_{j^{i}_1})\ldots s^{-}(m_{j^{i}_2})m_{j^{i}_2}y_im_{j^{i}_1}s^{-}(m_{j^{i}_1})\ldots s^{+}(m_{j^{i}_2})\]
to be the red path from $s^{+}(m_{j^{i}_1})$ to $s^{+}(m_{j^{i}_2})$ which contains
$y_i$ and is obtained from $\mathcal C^{i-1}$ by deleting the red edges $m_{j^{i}_1}s^{+}(m_{j^{i}_1})$ and
$m_{j^{i}_2}s^{+}(m_{j^{i}_2})$ and adding the red edges $m_{j^{i}_2}y_i$ and $y_im_{j^{i}_1}$. Let $Q^i_2$ be a path inside $S_{i}$ from $u^{i}_1$ to $u^{i}_2$ of length at most $2\ln n$ consisting solely of edges in $R^{i}$ (such a path exists by the choice of $S_{i}$). Set $\mathcal C^{i}$ to be the red cycle formed by joining the red paths $Q^i_1$ and $Q^i_2$ with the red edges $s^{+}(m_{j^{i}_1})u^{i}_1$ and $s^{+}(m_{j^{i}_2})u^{i}_2$. To end the routine, set $\text{EXP}^{i}:=\text{EXP}^{i-1}\setminus V(\mathcal C^{i})$.
\vspace{-8pt}
\begin{flushright}
\textbf{(End of Routine)}
\end{flushright}
\vspace{-4pt}
To finish the algorithm set $\text{EXP}_3:=\text{EXP}^{s}$ and $\mathcal{C}_3:=\mathcal C^{s}$.
\vspace{-14pt}
\begin{flushright}
\textbf{(End of Algorithm)}
\end{flushright}
\vspace{-4pt}

We make a few observations about the procedure above which will be useful for later:
\begin{enumerate}[(O1)]
\item For every $i\in [s]$ we have $|\text{EXP}^{i-1}\setminus \text{EXP}^{i}|\le 2\ln n + 1$ since $\text{EXP}^{i-1}\setminus \text{EXP}^{i}=V(Q^i_2)$ which has size at most $2\ln n +1$. Moreover, since by III.1(b) we have $d_{R_3}(v,\text{EXP}_2)\ge 2(\ln n)^{0.4}$ for every $v\in \text{EXP}_2$, every vertex in $\text{BAD}_{i}$ must have at least one red neighbour in $\text{EXP}_2\setminus \text{EXP}^{i-1}$. Thus, we have $|\text{BAD}_{i}|\le |\text{EXP}_2\setminus \text{EXP}^{i-1}|\cdot \Delta(R_3)\le (2\ln n+1)(i-1)\cdot \Delta(R_3)$.

\item For every $i\in [s+1]$ we have $N_{W^{i-1}}(v,\text{EXP}^{i-1})=N_{W_3}(v,\text{EXP}^{i-1})$ provided $v\in V(\mathcal C_2)\setminus \left(\bigcup_{k=1}^{i-1}A_k\right)$. Indeed, this holds since before Part $i$ of the routine we only recolour edges between the sets $A_{k}$ and $\text{EXP}^{k-1}$ for $k\in [i-1]$.

\item For every $i\in [s+1]$ and every $j\in J\setminus \left(\bigcup_{k=1}^{i-1}\{j^{k}_1,j^{k}_2\}\right)$ the path $P_j$ is still a path in the red cycle $\mathcal C^{i-1}$. Indeed, this follows by induction on $i$ using the facts that $m_{j^{k}_1}\in M_{j^{k}_1}$, $m_{j^{k}_2}\in M_{j^{k}_2}$ and that the sets $J_{k}$ are disjoint for $k\in [i-1]$.

\item The sets $(A_i)_{i\in [s]}$ are disjoint and $N_{W^{i-1}}(v,\text{EXP}^{i-1})=N_{W_3}(v,\text{EXP}^{i-1})$ for every $v\in A_i$ and every $i\in [s]$. Indeed, (O3) together with the fact that $m_{j}\in M_{j}$ for every $j\in J_i$ ensures that $A_i\subseteq \bigcup_{j\in J_i} V(P_j)$ for every $i\in [s]$. Since the sets $(J_i)_{i\in [s]}$ are disjoint, the first observation follows. The second observation follows from the first one together with (O2).

\end{enumerate}

In the next claim we prove that some properties which are assumed in the algorithm hold whp.
\begin{claim}
\label{claim:PhaseIVterminates}
All of the following properties hold whp:
\begin{enumerate}[$(i)$]
\item There exist disjoint subsets $(J_i)_{i\in [s]}$ of $J$ of size $10^3(\ln n)^{0.45}$ such that $d_{R_3}(y_i,M_j)>0$ for every $j\in J_i$.
\item For every $i\in [s]$ there exists $S_i\subseteq \text{GOOD}_i$ of size at least $\frac{1}{240}n(\ln n)^{-0.45}$ such that $R^{i-1}[S_i]$ is a connected graph of diameter at most $2\ln n$.
\item For every $i\in [s]$, after recolouring the edges in $W^{i-1}$ between the sets $A_i$ and $S_i$, there exist $u^{i}_1,u^{i}_2\in S_i$ and distinct indices $j^{i}_1,j^{i}_2\in J_i$ such that $s^{+}(m_{j^{i}_1})u^{i}_1\in R^{i}$ and $s^{+}(m_{j^{i}_2})u^{i}_2\in R^{i}$.
\end{enumerate}
\end{claim}
\begin{proof}[Proof of Claim \ref{claim:PhaseIVterminates}]
We show first that whp $(i)$ holds. Let $G_{\text{aux}}$ be the bipartite graph with parts $[s]$ and $J$ and edge set $\{ij: i\in [s], j\in J, d_{R_3}(y_i,M_j)>0\}$. We want to show that whp there are disjoint subsets $(J_i)_{i\in [s]}$ of $J$ of size $10^3(\ln n)^{0.45}$ such that $J_i\subseteq N_{G_{\text{aux}}}(i)$ for every $i\in [s]$. In light of Lemma \ref{lemma: star matching} it suffices to show that whp $|N_{G_{\text{aux}}}(I)|\ge 10^3(\ln n)^{0.45}|I|$ for every $I\subseteq [s]$. With this in mind, suppose that $I\subseteq [s]$ is such that $|N_{G_{\text{aux}}}(I)|<10^3(\ln n)^{0.45}|I|$, and set $X:=\{y_i:i\in I\}$ and $Y:=\bigcup_{j\in N_{G_{\text{aux}}}(I)}M_j$. Note that we have
\[|X\cup Y|= |X|+98|N_{G_{\text{aux}}}(I)|< 10^5(\ln n)^{0.45}|X|\]
and, using III.2(b) and the fact that $N_{R_3}(X)\cap \mathcal M_2\subseteq Y$, we also have
\[e_{R_3}(X\cup Y)\ge ((\ln n)^{0.5}-400)|X|>\frac{(\ln n)^{0.5}-400}{10^5(\ln n)^{0.45}}|X\cup Y|>\frac{(\ln n)^{0.05}}{10^6}|X\cup Y|.\]
By $(P3)$ of Lemma \ref{lemma:propertiesofGnp} we see then that whp $|X\cup Y|>\frac{1}{10^6}(\ln n)^{0.05}\cdot \frac{1}{20}n(\ln n)^{-1}=\frac{1}{2\times 10^7}n(\ln n)^{-0.95}$. But in that case we have
\[|\text{SMALL}_2|\ge |X|>\frac{|X\cup Y|}{10^5(\ln n)^{0.45}}>\frac{n}{2\times 10^{12}(\ln n)^{1.4}}\]
which contradicts III.2(a). We conclude that whp $(i)$ holds.

Next we show that whp $(ii)$ also holds. Indeed, assuming that $\Delta(R_3)\le 40\ln n$ (which holds whp after Phase III, by N.1), we have by (O1), III.1(a) and III.2(a) that
\begin{align*}
|\text{GOOD}_{i}|&\ge |\text{EXP}^{i-1}|-|\text{BAD}_{i}\cup N_{R_3}(\text{BAD}_{i})|\\
&\ge |\text{EXP}_2|-(2\ln n+1)\cdot(i-1)-|\text{BAD}_{i}|\cdot (1+40\ln n)\ge \left(1-\frac{1}{\ln n}\right)|\text{EXP}_1|.
\end{align*}
Thus, by II.1(c) there is always a set $S_{i}\subseteq \text{GOOD}_{i}$ with the desired properties.

Finally, we show that whp $(iii)$ also holds. By (O4), for every
$i\in [s]$ we have $d_{W^{i-1}}(u,S_{i})=d_{W_3}(u,S_{i})$ for every
$u\in A_i$. Moreover, since every vertex $u\in A_i$ is at distance
at most $2$ in $R_3$ from $y_i\in \text{SMALL}_2$, it follows from
III.2(c) and I.4 that
$d_{W^{i-1}}(u,S_{i})=d_{W_1}(u,S_{i})=(1-o(1))|S_i|$ for all but at
most $100$ vertices $u\in A_{i}$. Let $D_i$ be the set of these
``bad" vertices. For $j\in\{1,2\}$ and $i\in [s]$ let
$B^{i}_{j}\subseteq A_{i}\setminus D_i$ be sets of size at least
$400(\ln n)^{0.45}$ such that $B^{i}_{1}\cap B^{i}_{2}=\emptyset$
(this is possible since $|A_i|=|J_i|=10^3(\ln n)^{0.45}$ and
$|D_i|\le 100$). Thus, for any $i\in [s]$ and $j\in\{1,2\}$ the
probability that there is no edge in $R^i$ between $B^{i}_j$ and
$S_{i}$ is at most:
\[\Pr\left[\text{Bin}\left((1-o(1))|S_{i}|\cdot |B^{i}_{j}|,p\right)=0\right]= (1-p)^{(1-o(1))|S_{i}|\cdot |B^{i}_{j}|}\le e^{-\frac{\ln n}{n}\cdot \frac{1-o(1)}{240}n(\ln n)^{-0.45}\cdot 400(\ln n)^{0.45}}\le \frac{1}{n}.\]
Using III.2(a) and the union bound we see that the probability that for some $i\in [s]$ and $j\in\{1,2\}$ there is no edge in $R^{i}$ between $B^{i}_{j}$ and $S_{i}$  is at most $2s\cdot \frac{1}{n} = o(1)$. Since the sets $B^{i}_1$ and $B^{i}_2$ were chosen to be disjoint, we conclude that whp  $(iii)$ holds. This completes the proof of the claim.
\end{proof}

Assuming that the properties of Claim \ref{claim:PhaseIVterminates} hold, denoting by $R_4$, $W_4$ and $B_4$ the sets of red, white and blue edges at the end of this phase's algorithm, we show that whp the following technical conditions hold:
\begin{enumerate}
\item[IV.1] Properties of $\text{EXP}_3$:
\begin{enumerate}
\item $|\text{EXP}_3|\ge\left(1-\frac{1}{\ln n}\right)|\mathcal U|\ge (2-o(1))n(\ln n)^{-0.45}$.
\item for every $v\in \text{EXP}_3$ we have $d_{R_4\setminus R_1}(v,\text{EXP}_3)\ge (\ln n)^{0.4}$.
\end{enumerate}
\item[IV.2] In this phase only $o(n)$ edges of $W_3$ are recoloured red.
\end{enumerate}

The next few claims show that properties IV.1-IV.2 all hold whp, assuming that the properties of Claim \ref{claim:PhaseIVterminates} hold.

\begin{claim}
\label{claim:PropertyIV.1holds}
Properties IV.1(a) and IV.1(b) hold whp.
\end{claim}
\begin{proof}[Proof of Claim \ref{claim:PropertyIV.1holds}]
Note that by (O1), III.1(a) and III.2(a) we have:
\[|\text{EXP}_3|= |\text{EXP}_2|-\sum_{i=1}^{s}|\text{EXP}^{i-1}\setminus \text{EXP}^{i}|\ge \left(1-\frac{1}{(\ln n)^2}\right)|\mathcal U|-s\cdot(2\ln n+1)= \left(1-\frac{1}{\ln n}\right)|\mathcal U|\]
and so this together with I.1 shows that IV.1(a) holds.

We now show that IV.1(b) holds whp. Suppose $d_{R_4\setminus R_1}(v,\text{EXP}_3)<(\ln n)^{0.4}$ for some $v\in \text{EXP}_3$ and let $i\in [t]$ be the largest possible integer such that $d_{R_3\setminus R_1}(v,\text{EXP}^{i-1})\ge 2(\ln n)^{0.4}$ (this is well defined by III.1(b)).
Note that by the choice of $i$ we have $v\in \text{BAD}_{j}$ for every $j>i$. Moreover, since $R_4[\mathcal U]=R_3[\mathcal U]$, $\text{EXP}^{j-1}\setminus \text{EXP}^{j}=V(Q^{j}_2)\subseteq \text{GOOD}_{j}$ and $\text{GOOD}_{j}\cap N_{R_3}(\text{BAD}_j)=\emptyset$ for any $j>i$ we have
\[d_{R_4\setminus R_1}(v,\text{EXP}_3)=d_{R_3\setminus R_1}(v,\text{EXP}^{i-1})-\sum_{j=i}^{s}d_{R_3\setminus R_1}(v,V(Q^{j}_2))\ge 2(\ln n)^{0.4}-d_{R_3}(v,V(Q^{i}_2)).\]
Thus, since we assume that $d_{R_4\setminus R_1}(v,\text{EXP}_3)<(\ln n)^{0.4}$, we must have that $d_{R_3}(v,V(Q^{i}_2))>(\ln n)^{0.4}$. It is easy to see that this implies that there exist two cycles of length $O((\ln n)^{0.6})$ sharing only the vertex $v$ in the graph $R_3$ (since $Q^{i}_2$ is a path in $R_3$ of length at most $2\ln n+1$). However, if N.3 holds then this does not happen. Since N.3 holds whp we conclude that IV.1(b) holds whp as desired.
\end{proof}

\begin{claim}
\label{claim:PropertyIV.2holds}
Property IV.2 holds whp.
\end{claim}
\begin{proof}[Proof of Claim \ref{claim:PropertyIV.2holds}]
Recall from the algorithm that in this phase the only edges recoloured are the ones in $W_3$ between the sets $A_i$ and $S_i$ for each $i\in [s]$. Since for each $i\in [s]$ we have $|A_i|=|J_i|=10^3(\ln n)^{0.45}$ and $|S_i|\le |\mathcal U|\le 4n(\ln n)^{-0.45}$ by I.1, the total number of edges which are recoloured in this phase is at most $s\cdot 4\cdot 10^3n\le 8\cdot 10^3n^{2}e^{-(\ln n)^{0.4}}$, by III.2(a). Thus, using Chernoff (Lemma \ref{Che}) and the fact that $np\le 10\ln n$ we see that the probability that more than $1.6\cdot 10^5 ne^{-(\ln n)^{0.4}}\ln n$ edges in $W_3$ are recoloured red in this phase is at most:
\[\Pr\left[\text{Bin}\left(8\cdot 10^3 n^{2}e^{-(\ln n)^{0.4}},p\right)>1.6\cdot 10^5 ne^{-(\ln n)^{0.4}}\ln n\right]
<e^{-\frac{1}{3}\cdot 8\cdot 10^4ne^{-(\ln n)^{0.4}}\ln n}=o(1).\]
Since $1.6\cdot 10^5 ne^{-(\ln n)^{0.4}}\ln n=o(n)$, this concludes the proof of the claim.
\end{proof}

We have shown that whp at the end of this phase all of the properties IV.1-IV.2 hold. We shall assume henceforth that all these properties hold for the sets $R_4$, $W_4$ and $B_4$.

\subsection{Phase V}

In this phase we create a Hamilton cycle in the red graph by merging the red cycle
$\mathcal C_3$ with the set $\text{EXP}_3$, by recolouring red $o(n)$ edges. To this end, note first that $R_4[\text{EXP}_3]$ has a bounded number of connected components. Indeed, suppose $C$ is a connected component of $R_4[\text{EXP}_3]$. It follows from properties IV.1(a), IV.1(b), i. of II.1(c) and the fact that $R_4[\text{EXP}_3]=R_2[\text{EXP}_3]$, that the set $C$ has size $|C|>\frac{1}{6000}n(\ln n)^{-0.45}$ (since $N_{R_2[S]}(C)\subseteq C$). However, since by I.1, the set $\text{EXP}_3$ has size $|\text{EXP}_3|\le 4n(\ln n)^{-0.45}$, we can conclude that the graph $R_4[\text{EXP}_3]$ has at most $24000$ connected components. By recolouring red less than $24000$ white edges in $\text{EXP}_3$ we can then whp make the red graph in $\text{EXP}_3$ connected.

Afterwards, we consider two adjacent vertices $v,w$ in the red cycle $\mathcal C_3$ which have large white degree onto $\text{EXP}_3$. By recolouring edges between $v,w$ and $\text{EXP}_3$ we can then whp find $x\ne y\in \text{EXP}_3$ such that $vx$ and $wy$ are red edges. Finally, by recolouring red at most $|\text{EXP}_3|$ edges inside $\text{EXP}_3$ we can find a red Hamilton path from $x$ to $y$ inside the set $\text{EXP}_3$. This path together with the red path $\mathcal C_3\setminus\{vw\}$ and the red edges $vx$ and $wy$ then provides the desired red Hamilton cycle in $V$. The algorithm for this phase
is as follows:

\vspace{4pt}
{\bf Algorithm:} Let $C_1,\ldots,C_{\ell}$ be the connected components of the graph $R_4[\text{EXP}_3]$ where, as indicated above, we have $\ell\le 24000$ and $|C_i|>\frac{1}{6000}n(\ln n)^{-0.45}$ for every $i\in [\ell]$. For $1\le i<\ell$ recolour white edges between $C_i$ and $C_{i+1}$ one by one until exactly one edge is recoloured red for each such $i$. Claim \ref{claim:PhaseVterminates} ensures that this happens whp and we assume this henceforth. Note that this procedure makes the red graph in $\text{EXP}_3$ connected.

Next, let $v,w$ be adjacent vertices of $\mathcal C_3$ such that $d_{W_4}(v,\text{EXP}_3)\ge \frac{2}{3}|\text{EXP}_3|$ and $d_{W_4}(w,\text{EXP}_3)\ge \frac{2}{3}|\text{EXP}_3|$. Recolour edges in $W_4$ between $\{v,w\}$ and $\text{EXP}_3$ until there exist two edges $vx$ and $wy$ which are red, where $x,y\in \text{EXP}_3$ are distinct vertices. Claim \ref{claim:PhaseVterminates} ensures that whp such vertices $v$, $w$, $x$ and $y$ exist and we assume this henceforth. Now, set $e=xy$ and run the following routine:

{\bf Routine:} Consider the graph $H$ which is the current red graph on $\text{EXP}_3$. If $H\cup \{e\}$ does not contain a Hamilton cycle that uses the edge $e$, recolour $e$-boosters of $H$ which are white edges one by one until one of them is recoloured red. Repeat this procedure until the graph $H\cup\{e\}$ considered contains a Hamilton cycle which uses the edge $e$.
\vspace{-14pt}
\begin{flushright}
\textbf{(End of Routine)}
\end{flushright}
\vspace{-4pt}

Claim \ref{claim:PhaseVterminates} ensures that whp this procedure is successful and we assume this henceforth. A Hamilton cycle in $H\cup \{e\}$ which uses the edge $e$ then provides a red Hamilton path in $\text{EXP}_3$ from $x$ to $y$. This red path together with the red path $\mathcal C_3\setminus \{vw\}$ and the red edges $vx$ and $wy$ forms the desired red Hamilton cycle in $V$.
\vspace{-10pt}
\begin{flushright}
\textbf{(End of Algorithm)}
\end{flushright}
\vspace{-4pt}

In the next claim we prove that some properties which are assumed in the algorithm hold whp.
\begin{claim}
\label{claim:PhaseVterminates}
All of the following properties hold whp:
\begin{enumerate}[$(i)$]
\item For every $1\le i<\ell$, if we recolour all the white edges between $C_i$ and $C_{i+1}$ then at least one edge is recoloured red.
\item There exist adjacent vertices $v,w$ in $\mathcal C_3$ such that $d_{W_4}(v,\text{EXP}_3)\ge \frac{2}{3}|\text{EXP}_3|$ and $d_{W_4}(w,\text{EXP}_3)\ge \frac{2}{3}|\text{EXP}_3|$. Moreover, after recolouring all the edges in $W_4$ between $\{v,w\}$ and $\text{EXP}_3$, there exist distinct vertices $x,y\in \text{EXP}_3$ such that $vx$ and $wy$ are red edges.
\item At any point during the routine, if $H$ is the graph considered and if we recolour all the $e$-boosters of $H$ which are white edges then there will be one which is recoloured red.
\end{enumerate}
\end{claim}

\begin{proof}
We start by showing that $(i)$ holds whp. Indeed, this follows from Lemma \ref{pseudorandom} since for every $i$ we have $|C_i|>\frac{1}{6000}n(\ln n)^{-0.45}$.

Next we show that $(ii)$ holds whp. Recall from I.4 that for every $u\in V(\mathcal C_1)$ we have $d_{W_1}(u,\mathcal U)=(1-o(1))|\mathcal U|$. Moreover, since $|\text{EXP}_3|=(1-o(1))|\mathcal U|$ by IV.1(a) it follows that for every $u\in V(\mathcal C_1)$ we have $d_{W_1}(u,\text{EXP}_3)=(1-o(1))|\text{EXP}_3|$. Recall now that in Phase II there were no edges recoloured between $\text{EXP}_3$ and $V(\mathcal C_1)$ and that in Phases III and IV the number of vertices $u\in V(\mathcal C_1)$ for which we recoloured edges touching $u$ and vertices of $\text{EXP}_3$ is $o(n)$. Thus, it follows that for all but $o(n)$ vertices $u\in V(\mathcal C_1)$ we have $d_{W_4}(u,\text{EXP}_3)=(1-o(1))|\text{EXP}_3|$. Since $|V(\mathcal C_3)\setminus V(\mathcal C_1)|=o(n)$ we can find two vertices $v,w\in V(\mathcal C_1)$ which are adjacent in $\mathcal C_3$ and for which $d_{W_4}(v,\text{EXP}_3)\ge \frac{2}{3}|\text{EXP}_3|$ and $d_{W_4}(w,\text{EXP}_3)\ge \frac{2}{3}|\text{EXP}_3|$, as claimed.

Partition now the set $\text{EXP}_3$ into two sets $A_v$ and $A_w$ of size as equal as possible. If we recolour all the edges in $W_4$ between $\{v,w\}$ and $\text{EXP}_3$ then the probability that afterwards either there is no red edge $vx$ with $x\in A_v$ or $wy$ with $y\in A_w$ is at most
\[2\Pr\left[\text{Bin}\left(\left(\frac{1}{6}-o(1)\right)|\text{EXP}_3|,p\right)=0\right]=2(1-p)^{\left(\frac{1}{6}-o(1)\right)|\text{EXP}_3|}\le e^{-\left(\frac{1}{6}-o(1)\right)|\text{EXP}_3|p}=o(1)\]
by IV.1(a). We conclude that property $(ii)$ of the claim holds whp.

Finally we show that $(iii)$ also holds whp. Note first that the routine can be executed at most $|\text{EXP}_3|$ times since each time the size of a longest path in the graph considered increases by at least one. Now, let $H$ be one of the graphs considered during the routine. Note that the edge set $E(H)$ of $H$ is of the form $E(H)=K_1\cup K_2$ where $K_1=R_4[\text{EXP}_3]$ and $K_2\subseteq \text{EXP}_3$ is a set of size $|K_2|<|\text{EXP}_3|+\ell$. Indeed, $K_2=E(H)\setminus K_1$ consists of the $\ell-1$ red edges added in this phase connecting the connected components of $K_1$ and at most $|\text{EXP}_3|$ red edges added during the routine. Moreover, note that the graph $H$ is connected since the red graph on $\text{EXP}_3$ is already connected when the routine is executed. Since $R_4[\text{EXP}_3]=R_2[\text{EXP}_3]$ and since the graph $(R_4\setminus R_1)[\text{EXP}_3]$ has minimum degree at least $(\ln n)^{0.4}$ by IV.1(b), it follows from i. of II.1(c) that for any set $X\subseteq \text{EXP}_3$ of size $|X|\le \frac{1}{6000}n(\ln n)^{-0.45}$ we have $|N_{H}(X)\cup X|\ge 5|X|$. Thus, again using the fact that $R_4[\text{EXP}_3]=R_2[\text{EXP}_3]$ and that $\ell\le 24000$ we conclude from II.1 (d) that the number of $e$-boosters for $H$ in the set $W_2[\text{EXP}_3]$ is at least $10^{-8}n^2(\ln n)^{-0.9}$.

Recall now that in Phases III and IV no white edges inside $\text{EXP}_3$ were recoloured and so $W_4[\text{EXP}_3]=W_2[\text{EXP}_3]$. Moreover, as indicated above, in Phase V we recolour red less than $\ell +|\text{EXP}_3|$ white edges inside $\text{EXP}_3$.
The probability that at least $10^{-9}n^2(\ln n)^{-0.9}$ white edges are recoloured blue during Phase V is then at most the probability that at most $\ell+|\text{EXP}_3|-1\le 5n(\ln n)^{-0.45}$ white edges are recoloured red before $10^{-9}n^2(\ln n)^{-0.9}$ white edges are recoloured blue, which is at most
\begin{align*}
&\sum_{j=0}^{5n(\ln n)^{-0.45}}\binom{10^{-9}n^2(\ln n)^{-0.9}+j}{j}(1-p)^{10^{-9}n^2(\ln n)^{-0.9}}p^{j}\\
\le\; & O(n(\ln n)^{-0.45})\cdot\binom{O(n^2(\ln n)^{-0.9})}{5n(\ln n)^{-0.45}}\cdot p^{5n(\ln n)^{-0.45}}\cdot e^{-p\cdot \Omega(n^2(\ln n)^{-0.9})}\\
\le\; & O(n(\ln n)^{-0.45})\cdot (O((\ln n)^{0.55}))^{5n (\ln n)^{-0.45}}\cdot e^{-\Omega(n(\ln n)^{0.1})}=o(1)\ .
\end{align*}
Thus, whp less than $10^{-9}n^2(\ln n)^{-0.9}$ white edges are recoloured blue in this phase. Thus, since the set $W_4[\text{EXP}_3]$ contains at least $10^{-8}n^2(\ln n)^{-0.9}$ $e$-boosters for $H$ we conclude that whp at any point in the routine there are at least $9\cdot 10^{-9}n^2(\ln n)^{-0.9}$ $e$-boosters for $H$ which are white edges (at that point). The probability that none of these $e$-boosters is recoloured red is then at most
\[\Pr\left[\text{Bin}(9\cdot 10^{-9}n^2(\ln n)^{-0.9},p)=0\right]=(1-p)^{9\cdot 10^{-9}n^2(\ln n)^{-0.9}}\le e^{-p\cdot 9\cdot 10^{-9}n^2(\ln n)^{-0.9}}\le e^{-9\cdot 10^{-9}n(\ln n)^{0.1}}\,.\]
Thus, since the routine is executed at most $\ell +|\text{EXP}_3|\le 5n(\ln n)^{-0.45}$ many times, we conclude by the union bound that the probability that at some point in the routine all the $e$-boosters of $H$ which are white edges (where $H$ is the graph being considered at that point) are recoloured blue is at most
\[5n(\ln n)^{-0.45}\cdot e^{-9\cdot 10^{-9}n(\ln n)^{0.1}}=o(1)\,,\]
and so property $(iii)$ holds whp as claimed.
\end{proof}

Assuming that the properties of Claim \ref{claim:PhaseVterminates} hold and denoting by $R_5$ the set of red edges at the end of this phase's algorithm, it is clear from it that the graph $R_5$ contains a Hamilton cycle.

We claim now that $|R_5\setminus R_4|< \ell+3|\text{EXP}_3|=o(n)$. Indeed, in the beginning of this phase's algorithm exactly $\ell-1$ edges are recoloured red in order to make the red graph in $\text{EXP}_3$ connected. Later, we recoloured edges between $\{v,w\}$ and $\text{EXP}_3$ and so at that point at most $2|\text{EXP}_3|$ edges are recoloured red. Finally, we recoloured one edge red each time the routine was executed. However, as indicated in the proof of Claim \ref{claim:PhaseVterminates} the routine can be executed at most $|\text{EXP}_3|$ times. Thus, we conclude that less than $\ell+3|\text{EXP}_3|=o(n)$ edges are recoloured red in this phase, as claimed.

To finish the proof of Theorem \ref{thm:main1} we show that $|R_5|=n+o(n)$. Indeed by I.2, II.4, III.4, IV.2 and by the previous paragraph we conclude that
\[|R_5|=|R_1|+|R_5\setminus R_1|=n+o(n)\,\]
as desired.

\section{Concluding remarks}

In this paper we introduced a new type of problems in random graphs,
where the goal is to expose a subgraph which possesses some target
property $\mathcal P$, by asking as few queries as possible. Note
that this problem is general and can be considered in any model of
random structures.

Although we chose to focus on the property of Hamiltonicity, our
proof method can be applied to prove analogous statements regarding
other interesting properties. For example, one can show that for
$p\geq \frac{\ln n+\omega(1)}{n}$, there exists an adaptive
algorithm, interacting with the probability space $\mathcal G(n,p)$,
which whp finds a matching of size $\lfloor n/2 \rfloor$ (a perfect
matching) after getting $n/2+o(n)$ positive answers.

Let us now show that one cannot get rid of the $o(n)$ term. More
precisely, we show that whp at least $n+\Omega\left(\sqrt{\frac{1}{p}}\right)$ positive
answers are needed in order to find a Hamilton cycle. In particular, if $p=\Theta\left(\frac{\log n}{n}\right)$ then this means that at least $\Theta\left(\sqrt{\frac{n}{\log n}}\right)$ extra positive answers are needed to find a Hamilton cycle. For this aim,
let $k:=k(n)$ be an integer and let $G$ be a non-Hamiltonian graph
on $n$ vertices with exactly $n+k$ edges. Suppose that there exist a
non-edge $xy$ of $G$ for which $G+xy$ is Hamiltonian, and observe
that $G$ contains at most two vertices of degree $1$ and no isolated
vertices. First, let us note that both $x$ and $y$ can have at most
one neighbor in $G$ which is of degree $2$. Indeed, every neighbor
of (say) $x$ which has degree exactly $2$ must be connected to $x$
on the Hamilton cycle created by adding $xy$ to $G$. Since $xy$ must
be an edge of this cycle, $x$ cannot have more than one such
neighbor. Second, we try to estimate the number of pairs $xy$ for
which both $x$ and $y$ have at most one neighbor of degree $2$ in
$G$. For this end, let $A$ denote the set of all vertices in $G$ of
degree distinct than $2$, and let $a:=|A|$. Since $G$ has exactly
$n+k$ edges, and since there are at most $2$ vertices of degree $1$,
it follows that $2+2(n-a)+3(a-2)\leq 2n+2k$. Therefore, we have
$a\leq 2k+4$. Next, since $2(n-a)+\sum_{v\in A}d_G(v)\leq 2n+2k$,
using the estimate we obtained on $a$ we conclude that
\[|N(A)|\leq \sum_{v\in
A}d_G(v)=O(k).\]
Thus, all in all, we have $|A\cup N(A)|=O(k)$. Observe now crucially that if $xy$ is a non-edge of $G$ for which $G+xy$ is Hamiltonian then $x$ and $y$ must be in $A\cup N(A)$. Indeed, if (say) $x\notin A$ then $x$ must have degree $2$ and so at least one of its neighbours lies in $A$ (as discussed above) and so $x\in N(A)$. Hence, there are $O(k^2)$ pairs $xy$
for which $G+xy$ might be Hamiltonian. Suppose now that we have an adaptive algorithm, interacting with the probability space $\mathcal G(n,p)$, which whp finds a Hamilton cycle after getting at most
$n+k+1$ positive answers. Let $G$ be the random graph obtained by the algorithm whose edges correspond to the positive answers until the step just before a Hamilton cycle is found. Note that by hypothesis whp $G$ has at most $n+k$ edges and so by the reasoning above, it follows that there are at most $O(k^2)$ possible non-edges of $G$ which can be queried to turn $G$ into a Hamiltonian graph. However, if $k=o\left(\sqrt{\frac{1}{p}}\right)$, since $k^2p=o(1)$, by conditioniong on what the graph $G$ can be, we see by Markov's
inequality that whp none of these pairs of vertices obtain positive answers even if we query all
of them. Thus, we conclude that no such algorithm exists.

Note that even though the general setting we introduced appears to be new,
there has been some previous work of this flavor in the literature,
albeit inexplicit. For example, the DFS-based argument of
\cite{KS13} indicates that in the super-critical regime
$p=\frac{1+\varepsilon}{n}$, $\Theta(\varepsilon)n$ positive answers
suffice to uncover typically a connected component of size
proportional to $\varepsilon n$, and this is clearly optimal. The
analysis from \cite{KS13} also gives an adaptive algorithm for
finding a path of length $\Theta(\varepsilon^2)n$ (which is whp the
asymptotic order of magnitude of a longest path in such a random
graph) after querying $\Theta(\varepsilon)n$ edges successfully.
What matters here is the dependence on $\varepsilon$, and the above
stated algorithmic bound is above the trivial lower bound by the
$1/\varepsilon$ factor. In a companion paper \cite{FKSV2} we show
that this gap cannot be bridged and the algorithm from \cite{KS13} is essentially
(up to a $\log(1/\varepsilon)$ factor) best possible.

Another natural instance of the setting promoted in this paper is when the
target property ${\cal P}$ is the containment of a fixed graph $H$.
In this case, the obvious lower bound for the total number of queries needed is of order $1/p$. It appears
that the form of the optimal bound for the number of queries may
depend heavily on the value of $p$.
 Consider for example the case
$H=K_3$. For constant $p$, one can just query the pairs in
$\omega(1)$ pairwise disjoint triples of vertices to find w.h.p. a
copy of the triangle. However, say, for the case $p=n^{-1/2}$ the
right bound seems to be around $n^{3/4}$. Indeed, a simple algorithm
asking a bit more than $n^{3/4}$ queries would be first to query
pairs containing a fixed vertex $v$ till $\omega(n^{1/4})$ edges
touching $v$ are found -- this would take $\omega(n^{3/4})$ queries.
Querying now all pairs between the other points of these edges
uncovers w.h.p. an edge $(u,w)$ closing a triangle with $v$. For the
lower bound, one can argue that having $o(n^{1/4})$ positive answers
on the board produces only $o(n^{1/2})$ pairs of vertices at
distance two, and w.h.p. none of these pairs will show up in the
random graph to close a desired triangle. This argument has certain
similarities to the lower bound for avoidance of a given graph in
Achlioptas processes given in \cite{KLS09}. Of course the case of
triangles appears to be relatively easy, and we expect much more
involved analysis for a general $H$.

\medskip
\textbf{Acknowledgements.} We would like to thank the anonymous referees for carefully reading our paper as well as for their helpful comments. Parts of this work were carried out when the second author visited the Institute for Mathematical Research (FIM) of ETH Z\"urich, and also when the third author visited the School of Mathematical Sciences of Tel Aviv University, Israel. We would like to thank both institutions for their hospitality and for creating a stimulating research environment.

\end{document}